\documentclass[12pt]{amsart}%
\usepackage{amsfonts}
\usepackage{amsmath}
\usepackage{amssymb}
\usepackage{amsmath}

\allowdisplaybreaks[4]
\usepackage{graphicx}%
\usepackage{amsmath}
\allowdisplaybreaks[4]
\makeatletter
\@addtoreset{equation}{section}
\makeatother

\numberwithin{equation}{section}

\newtheorem{theorem}{Theorem}[section]

\newtheorem{lemma}[theorem]{Lemma}

\newtheorem{proposition}[theorem]{Proposition}
\newtheorem{remark}[theorem]{Remark}

\newcommand\R{\mathbb R}

\newcommand\mbb\mathbb
\newcommand\mbf\mathbf
\newcommand\mcal\mathcal
\newcommand\mfrak\mathfrak
\newcommand\mrm\mathrm
\newcommand\msf\mathsf
\renewcommand\a\alpha
\renewcommand\b\beta
\newcommand\g\gamma
\newcommand\G\Gamma
\renewcommand\d\delta
\newcommand\D\Delta
\newcommand\e\varepsilon
\newcommand\z\zeta
\renewcommand\t\theta
\newcommand\Th\Theta
\newcommand\la\lambda
\newcommand\La\Lambda
\newcommand\s\sigma
\newcommand\si\varsigma
\newcommand\Si\Sigma
\newcommand\ups\upsilon
\newcommand\U\Upsilon
\newcommand\ph\varphi
\renewcommand\o\omega
\renewcommand\O\Omega
\newcommand\wt\widetilde
\newcommand\wh\widehat
\newcommand\ol\overline
\newcommand\ul\underline
\newcommand\mr\mathring
\newcommand\ub\underbrace
\newcommand\pa\partial
\newcommand\n\nabla
\newcommand\fa\forall
\newcommand\ex\exists
\newcommand\es\emptyset
\newcommand\wk\rightharpoonup
\newcommand\inc\hookrightarrow
\newcommand\linf\varliminf
\newcommand\lsup\varlimsup
\newcommand\os\overset
\newcommand\us\underset
\newcommand\sr\stackrel
\newcommand\Ot\Leftarrow
\newcommand\To\Rightarrow
\newcommand\map\mapsto
\newcommand\ot\leftarrow
\newcommand\lot\longleftarrow
\newcommand\lto\longrightarrow
\newcommand\tot\leftrightarrow
\newcommand\ltot\longleftrightarrow
\newcommand\sm\backslash
\renewcommand\Cup\bigcup
\renewcommand\Cap\bigcap
\newcommand\sub\subset
\newcommand\Sub\Subset
\newcommand\sne\subsetneq
\newcommand\bus\supset
\newcommand\Bus\Supset

\newcommand\eq\equiv
\newcommand\ox\otimes
\newcommand\Ox\bigotimes
\newcommand\pl\oplus
\newcommand\Pl\bigoplus
\newcommand\x\times
\renewcommand\c\circ
\newcommand\q\quad
\renewcommand\l\left
\renewcommand\r\right
\newcommand\fr\frac

\makeatletter
\def\@makefnmark{}
\makeatother

\begin{document}
\title[the Hardy-type mean field equation]{Quantitative properties of the Hardy-type mean field equation}
\author{Lu Chen}
\address{Key Laboratory of Algebraic Lie Theory and Analysis of Ministry of Education, School of Mathematics and Statistics, Beijing Institute of Technology, Beijing 100081, P. R. China;
MSU-BIT-SMBU Joint Research Center of Applied Mathematics, Shenzhen MSU-BIT
University, Shenzhen 518172, China}
\email{chenlu5818804@163.com}
\author{Bohan Wang}
\address{Key Laboratory of Algebraic Lie Theory and Analysis of Ministry of Education, School of Mathematics and Statistics, Beijing Institute of Technology, Beijing
100081, PR China}
\email{wangbohanbit2022@163.com}
\author{Chunhua Wang$^{\dagger}$}
\address{School of Mathematics and Statistics and Key Laboratory of Nonlinear Analysis and Applications, Ministry of Education; Hubei key Laboratory of Mathematical Sciences, Wuhan, 430079, P. R. China}
\email{chunhuawang@ccnu.edu.cn}
\thanks{$\dagger$ Corresponding author.\\
This paper was partly supported by the National Key Research and Development Program (No.
2022YFA1006900) and National Natural Science Foundation of China (No. 12271027, No. 12471106).}
\maketitle

\begin{abstract}
 In this paper, we consider the following Hardy-type mean field equation

\[
\left\{ {\begin{array}{*{20}{c}}
   { - \Delta  u-\frac{1}{(1-|x|^2)^2}  u  = \lambda e^u}, & {\rm in} \ \ B_1,  \\
   {\ \ \ \ u = 0,} &\  {\rm on}\  \partial B_1,
\end{array}} \right.
\]
\[\]
where $\lambda>0$ is small and $B_1$ is the standard unit disc of $\R^2$. Applying the moving plane method of hyperbolic space and the accurate expansion of heat kernel on hyperbolic space, we establish the radial symmetry and Brezis-Merle lemma for solutions of Hardy-type mean field equation. Meanwhile, we also derive the quantitative results for solutions of Hardy-type mean field equation, which improves significantly the compactness results for the classical mean field equation
obtained by Brezis-Merle and  Li-Shafrir. Furthermore, applying the local Pohozaev identity from scaling,
blow-up analysis and a contradiction argument,
we prove that the solutions are unique when $\lambda$ is sufficiently close to 0.

\medskip

\maketitle {\small {\bf Keywords:}~blow-up analysis; Hardy-Moser-Onofri inequality; mean field equation; quantitative analysis; uniqueness.
 \\

{\bf 2010 MSC.} 34E05; 34D05; 35B06; 35B09.  }

\end{abstract}

\section{Introduction\label{introduction}}
Let $B_1$ denote the standard unit disc in $\R^2$. The classical Moser-Trudinger inequality \cite{Moser1,Trudinger} states that for $\alpha\leq 4\pi$,

\begin{equation}
\label{MT}
\mathop {\sup }\limits_{\ \|\nabla u\|_2\leq1}
\int_{B_1} e^{\alpha u^2} dx<+\infty, \quad \forall\; u \in  W_0^{1,2}(B_1),
\end{equation}
\[\]
which serves as a fundamental tool in the field of nonlinear analysis and geometric analysis. As $\alpha=4\pi$, the inequality \eqref{MT} is sharp and $4\pi$ could not be replaced by any larger constant. The slightly weaker form of \eqref{MT}

\[ \mathop {\inf }\limits_{u \in W_0^{1,2}(B_1)}\frac {1}{16\pi} \int_{B_1} |\nabla u|^2 dx - \log \int_{B_1} e^u dx> -\infty, \quad \forall\; u\in W^{1,2}_0(B_1)\]
\[\]
has been widely applied  in addressing issues related to prescribed Gaussian curvature and mean field equations, which is called the Moser-Onofri inequality (see \cite{Bec,Ono}).

\medskip
Another well-known inequality in nonlinear functional analysis is the Hardy inequality which can be stated as

\begin{equation}
\label{Hardy}
    \int_{B_1} |\nabla u|^2 dx - \int_{B_1} \frac {u^2}{(1-|x|^2)^2} dx \ge 0,\quad \forall\; u\in W^{1,2}_0(B_1).
\end{equation}
\[\]
The inequality \eqref{Hardy} is sharp (see \cite{Marcus}), namely, for any $\lambda > 1$,

$$\inf_{u\in W_0^{1,2}(B_1)}\left[\int_{B_1} |\nabla u|^2 dx - \lambda\int_{B_1} \frac {u^2}{(1-|x|^2)^2} dx\right] = -\infty.$$
\[\]
Later, Brezis and Marcus \cite{Brezis-Mar} proved that there exists a positive constant $C > 0$  such that

\begin{equation*}
    \label{Brz}\int_{B_1}|\nabla u|^2dx-\int_{B_1}\frac{u^2}{(1-|x|^2)^2}dx\geq C\int_{B_1} u^2dx,
\quad\forall u\in W_0^{1,2}(B_1).
\end{equation*}
\[\]
If we denote by $\mathcal{H}$ the completion of $C_c^{\infty}(B_1)$ under the norm $\Big(\int_{B_1}\big(|\nabla u|^2-\frac{u^2}{(1-|x|^2)^2}\big)dx\Big)^{\frac{1}{2}}$, then the above inequality still holds for $u\in \mathcal{H}$.
In \cite{Wang-Ye}, using blow-up analysis, Wang and Ye
studied the sharp Hardy-Trudinger-Moser inequality. Indeed, they showed
\begin{equation}
\label{HMT}
    \sup_{\int_{B_1}(|\nabla u|^2-\frac{u^2}{(1-|x|^2)^2})dx\leq 1}\int_{B_1}\exp(4\pi u^2)dx<+\infty.
\end{equation}
\[\]
It is worth noting that the space  $\mathcal{H}$  is locally embedded into $W^{1,2}(B_1)$. By Hardy-Trudinger-Moser inequality, one can deduce the following Hardy-Moser-Onofri inequality:

\begin{equation}\label{HMO}
\inf_{u\in \mathcal{H}}\Bigg[\frac {1}{16\pi} \int_{B_1} |\n u|^2 dx - \frac {1}{16\pi} \int_{B_1} \frac {u^2}{(1-|x|^2)^2}dx -  \log \int_{B_1}e^u dx\Bigg] > -\infty.
\end{equation}
\[\]
The critical points of this inequality \eqref{HMO} satisfy the following Hardy-type mean field equation:

 \begin{equation}\label{mfeqH}
\left\{ {\begin{array}{*{20}{c}}
   { \ \ - \Delta  u-\frac{1}{(1-|x|^2)^2}  u  = \frac{\rho e^u}{\int_{B_1}  e^udx} ,} & {\text{in} \ B_1, } \\
   u\in \mathcal{H},\\
   {u = 0,} & {\ \ \text{on} \ \partial B_1}
\end{array}} \right.
\end{equation}
\[\]
with $\rho=8\pi$.
\medskip

The classical mean field equation

 \begin{equation}\label{mfeq}
\left\{ {\begin{array}{*{20}{c}}
   { \ \ - \Delta  u  = \frac{\rho e^{u}}{\int_{\Omega}  e^{u} dx} }, & {\text{in} \ \Omega\subseteq \mathbb{R}^2, } \\
   {u = 0,} & {\  \text{on} \ \partial \Omega}
\end{array}} \right.
\end{equation}
\[\]
 appears in multiple contexts, including conformal geometry (see \cite{Aubin}), statistical mechanics (see \cite{Caglioti1,Caglioti2}), and various other fields of applied mathematics (see  \cite{Bebernes, Chandrasekhar, Gelfand,Murrey}).  Across all these areas, there is a significant interest in constructing solutions that exhibit blow-up behavior and concentrate  at specific points, the locations of which provide valuable insights into the geometric and physical characteristics of the problem being studied.
\vskip0.1cm

To investigate blow-up solutions for the mean field equation \eqref{mfeq}, a crucial approach is to establish its quantitative properties. This work can be traced back to
 the foundational work of Brezis and Merle in \cite{Brezis}. In the past few decades, numerous researchers, including Nagasaki-Suzuki \cite{Nagasaki}, Li-Shafrir \cite{Li2} and Ma-Wei \cite{Wei3}, have conducted extensive studies on the quantitative properties of mean field equation \eqref{mfeq}. We also notice that the quantitative properties for high-order mean field equation and $n$-Laplace mean field equation have also been established in (\cite{CLW,Martinazzi,Wei1}).
\vskip0.1cm

In summary, to the best of our knowledge, the quantitative analysis for Hardy-type mean field equation \eqref{mfeqH} remains unexplored. The absence of Brezis-Merle lemma for
equation \eqref{mfeqH} and the presence of critical Hardy singularity bring substantial challenge in studying the related issues. In this paper, we tackle these challenges and obtain the following result:
\vskip0.1cm

\begin{theorem}
\label{Th1.1th}
Assume that $u_\lambda\in \mathcal{H}$ satisfies equation

\begin{equation}\label{quan}
\left\{ {\begin{array}{*{20}{c}}
  \medskip
   { - \Delta  u-\frac{1}{(1-|x|^2)^2}  u  = \lambda e^u,\ \ \  {\rm in} \ B_1,}  \\
     \medskip
   {0<C_1\leq\int_{B_1}  {\lambda e^u dx}\leq C_2,}\\
   {\ \ \ \ \ \ \ \   \ \ \ \ \ \ \ \ \ \ u = 0,\ \ \ \ \ \ \ \ \ {\rm on} \ \partial B_1.}
\end{array}} \right.
\end{equation}
\[\]
Then $u_\lambda$ is positive and radially decreasing. When $\lambda\rightarrow 0$, the solution $u_\lambda$ must blow up at the origin. Furthermore, we have

$$
\int_{B_1}  {\lambda e^{u_\lambda} dx}\rightarrow 8\pi
$$
and
\begin{equation}\label{quantitative formula}
u_\lambda(x)\rightarrow  8\pi G( x,0)\ \ \  {\rm in} \ \  C_{loc}^1(B_1 \backslash \{0\}),
\end{equation}
\[\]
where $G(x,y)$ denotes the Green's function of $-\Delta-\frac{1}{(1-|x|^2)^2}$ and satisfies the equation

\begin{equation}\label{Green fcn}
\left\{ {\begin{array}{*{20}{c}}
\medskip
   {\big( - \Delta -\frac{1}{(1-|x|^2)^2} \big )G(x,y)
     =\delta_x(y)}, & { \ {\rm in}\ B_1,}  \\
   {\ \ G(x,y)=  0 }, & { \ \ \  {\rm on}\ \partial B_1.}  \\
\end{array}} \right.
\end{equation}
\end{theorem}
\medskip

\medskip
It is very effective to apply the moving plane method or the method of moving spheres to prove the symmetry and uniqueness of solutions(see \cite{CL,CLO,LZ}).
To prove Theorem \ref{Th1.1th}, applying the moving plane method of hyperbolic space, we first show the radial symmetry of solutions to
the Hardy-type mean field equation \eqref{quan}.
Then,
based on the accurate expansion of heat kernel on hyperbolic space,
we obtain a similar Brezis-Merle Lemma and prove that when $\lambda\rightarrow 0$, the solution $u_\lambda$ must blow up at the origin,

$$
\int_{B_{1}}  {\lambda e^{u_\lambda} dx}\rightarrow 8\pi.
$$
\[\]
Finally, we apply the local Pohozaev identity to prove

\begin{equation*} 
u_\lambda(x)\rightarrow  8\pi G( x,0)\ \ \  {\rm in} \ \  C_{loc}^1(B_1 \backslash \{0\}).
\end{equation*}
\medskip

Moreover, the uniqueness of solutions for elliptic equation is also a very
important topic.  Next we investigate the uniqueness of
solutions as $\lambda\rightarrow 0$ to the Hardy-type mean field equation \eqref{quan}
and obtain the following result.

\begin{theorem}\label{unique}
There exists $\lambda_0>0$ such that solutions of the Hardy-type mean field equation \eqref{quan} are unique for any $0<\lambda<\lambda_0.$
\end{theorem}

\begin{remark}
In \cite{CCL}, Chang-Chen-Lin proved the uniqueness of solutions to the classical mean filed equation
\eqref{mfeq} by applying the isoperimetric inequality and the non-degeneracy of solutions.
Here we will use a local Pohozaev identity from scaling, blow-up analysis combining a contradiction argument which are different of theirs.
\end{remark}

There are many uniqueness results of solutions
for elliptic equations with nonlinearities being power functions,
such as \cite{DLY,GILY,GPY,G-1993}. However, the uniqueness results for solutions of Hardy-type mean field equation \eqref{quan} is still unknown. To prove Theorem \ref{unique}, similar to \cite{LPP-2022}, we mainly apply a contradiction argument combining some local Pohozaev identities.
Specifically, assume that there exist two solutions $u^{1}_{\lambda}(y)$
 and $u^{2}_{\lambda}(y)$ of equation \eqref{quan}.
 Noting that we can show the
  solutions of equation \eqref{quan} are radial (see Proposition \ref{pro1}), by  Cauchy-initial uniqueness for ODE, we only need to prove $u^{1}_{\lambda}(0)=u^{2}_{\lambda}(0).$
 To this end, we wish to prove $u^{1}_{\lambda}(y)=u^{2}_{\lambda}(y)$
 in a ball with small radial and the center of the ball being at 0.
 Observing that up to hyperbolic translation, the solutions concentrate at 0, we only need to apply
  a Pohozaev identity of the solutions $u_{\lambda}$ from scaling, which makes our proof much simpler.
\vskip 0.1cm

We would like to point out that since there is a critical Hardy-singularity term in equation \eqref{quan}, different from \cite{LPP-2022}, we apply the local Pohozaev
identity on a small ball $B_{\delta r_{\lambda}}$ not in $B_{1}(0).$
In order to estimate $\frac{r^{(1)}_{\lambda}}{r^{(2)}_{\lambda}},$
we need an accurate relation between $\lambda$ and $c_{\lambda}.$
And we obtain a decay result for $u_{\lambda}(r_{\lambda}x)-c_{\lambda}$
in $B_{\frac{1}{r_{\lambda}}} \backslash B_{2R_\epsilon},$ where $R_{\epsilon}>0$ is small. This decay result is of great interest independently.
\medskip

The paper is organized as follows: In section \ref{s2}, we introduce some properties and Hardy-Littlewood-Sobolev inequality on hyperbolic spaces.
In section \ref{s5}, we separately accomplish the proof of Theorem \ref{Th1.1th}. And in section \ref{s6}, we consider  the asymptotic behavior of the solution $u_\lambda$ around the origin.  In  section \ref{s7}, we finish the proof of  Theorem \ref{unique}. To facilitate a more streamlined proof in section \ref{s7}, we present some necessary results in the appendix.

\section{Auxiliary Results}\label{s2}

\subsection{Hyperbolic space and Mobius transformation}\label{sect:2.1}
Denote $\mathbb{B}^2$ by the Poincare disk, which is a unit ball $B_1$ equipped with the usual Poincare metric $g=\left(\frac{2}{1-|x|^2}\right)^2g_e$, where $g_e$ represents
the standard Euclidean metric. The hyperbolic volume element can be written as $dV_{\mathbb{B}^2}=\left(\frac{2}{1-|x|^2}\right)^2dx$ and the geodesic distance from the origin to $x\in \mathbb{B}^2$ is given by $\rho(x)=\log \frac{1+|x|}{1-|x|}$. The associated Laplace-Beltrami operator $\Delta_{\mathbb{B}^2}$ and the gradient $\nabla_{\mathbb{B}^2}$ are given respectively by

$$
\Delta_{\mathbb{B}^2}=\frac{1-|x|^2}{4}
\Big((1-|x|^2)\Delta_{\mathbb{R}^2}\Big),\ \ \nabla_{\mathbb{B}^2}=\Big(\frac{1-|x|^2}{2}\Big)^2\nabla_{\mathbb{R}^2}.
$$
\medskip

For each $a\in \mathbb{B}^2$, we define the Mobius transformation $T_a$ by (see \cite{Ahlfors})

\begin{equation}\label{trans-1}
T_a(x)=\frac{|x-a|^2a-(1-|a|^2)(x-a)}{1-2x\cdot a+|x|^2a^2},
\end{equation}
\[\]
where $x\cdot a$ denotes the scalar product in $\mathbb{R}^2$. It is known that the volume element $dV_{\mathbb{B}^2}$ on $\mathbb{B}^2$ is invariant with the respect to the Mobius transformation, which deduces that for any $\varphi\in L^{1}(\mathbb{B}^2)$,
there holds

$$\int_{\mathbb{B}^2}|\varphi\circ T_a|dV_{\mathbb{B}^2}=\int_{\mathbb{B}^2}|\varphi|dV_{\mathbb{B}^2}.$$
\[\]
Furthermore, the commutativity of Mobius transformation  $T_a$ (hyperbolic translation) with the operator $-\Delta_{\mathbb{B}^2}$ still holds. That is to say that for any $\phi\in C^{\infty}_c(\mathbb{B}^2)$, there holds

$$\int_{\mathbb{B}^2}-\Delta_{\mathbb{B}^2}(\phi\circ T_a)(\phi\circ T_a)dV_{\mathbb{B}^2}=\int_{\mathbb{B}^2}(-\Delta_{\mathbb{B}^2}\phi)\circ T_a\cdot (\phi\circ T_a)dV_{\mathbb{B}^2}=\int_{\mathbb{B}^2}-\Delta_{\mathbb{B}^2}\phi\cdot \phi dV_{\mathbb{B}^2}.$$
\[\]
Using the Mobius transformation, we can define the geodesic distance from $x$ to $y$ in $\mathbb{B}^2$ as follows

$$
\rho(x,y)=\rho(T_{x}(y))=\rho(T_{y}(x))=\log \frac{1+T_{y}(x)}{1-T_{y}(x)}.
$$
\[\]
Also using the Mobius transformation again, we can define the convolution of measurable functions $f$ and $g$ on $\mathbb{B}^2$ by (see \cite{liu})

$$(f\ast g)(x)=\int_{\mathbb{B}^2}f(y)g(T_x(y))dV_{\mathbb{B}^2}(y),$$
\[\]
where $dV_{\mathbb{B}^2}(y)=\left(\frac{2}{1-|y|^2}\right)^{2}dy$.
\medskip

\medskip

\subsection{Hardy-Littlewood-Sobolev inequality on Poincare disk $\mathbb{B}^2$}

It is well known that for any $u\in W^{1,2}(\mathbb{B}^2)$ and $q\geq 2$, there holds

$$
\int_{\mathbb{B}^2}\big(|\nabla_{\mathbb{B}^2}u|^2-\frac{1}{4} |u|^2\big)dV_{\mathbb{B}^2}\geq C_{q}\Big(\int_{\mathbb{B}^2}|u|^qdV_{\mathbb{B}^2}\Big)^{\frac{2}{q}}.
$$
\[\]
By density argument, this inequality also holds for any function in the completion of $C_c^{\infty}(\mathbb{B}^2)$ under the norm $\int_{\mathbb{B}^2}\big(|\nabla_{\mathbb{B}^2} u|^2- \frac{1}{4}u^2\big)dV_{\mathbb{B}^2}$. If we define $f=(-\Delta_{\mathbb{B}^2}-\frac{1}{4})^{\frac{1}{2}}u$, then by duality,
the above inequality is equivalent to

$$
\Big(\int_{\mathbb{B}^2}|f|^{q'}dV_{\mathbb{B}^2}\Big)^{\frac{2}{q'}}\geq C_{q}\int_{\mathbb{B}^2}
\Big|(-\Delta_{\mathbb{B}^2}-\frac{1}{4})^{-\frac{1}{2}}f\Big|^2dV_{\mathbb{B}^2},
$$
\[\]
which can be written as

$$
\int_{\mathbb{B}^2}\int_{\mathbb{B}^2}
f(x)G(x,y)f(y)dV_{\mathbb{B}^2}(y)dV_{\mathbb{B}^2}(x)\leq C_{ q}^{-1}\Big(\int_{\mathbb{B}^2}|f|^{q'}dV_{\mathbb{B}^2}\Big)^{\frac{2}{q'}},
$$
\[\]
where $G(x,y)$ is the Green's function of the operator $-\Delta_{\mathbb{B}^2}-\frac{1}{4}$ on hyperbolic space $\mathbb{B}^2$, and $q'=\frac{q}{q-1}$.
Furthermore, we can also derive that for any $f$ and $g\in L^{q'}(\mathbb{B}^2)$, there holds

$$
\int_{\mathbb{B}^2}\int_{\mathbb{B}^2}f(x)G(x,y)g(y)dV_{\mathbb{B}^2}(y)dV_{\mathbb{B}^2}(x)\leq C_{ q}^{-1}\|f\|_{L^{q'}(\mathbb{B}^2)}\|g\|_{L^{q'}(\mathbb{B}^2)}.
$$
\[\]
If we define $I(f)=\int_{\mathbb{B}^2}G(x,y)f(y)dV_{\mathbb{B}^2}(y)$, then it follows that

$$
\|I(f)\|_{L^q(\mathbb{B}^2)}\leq C_{q}^{-1} \|f\|_{L^{q'}(\mathbb{B}^2)}.
$$
\medskip

\medskip

\subsection{Total geodesic line of two dimensional Hyperboloid model }

Let $\mathbb{R}^{2,1}=(\mathbb{R}^{2+1},g)$ be the Minkowski space, where the metric satisfies

$$
ds^2=dx_1^2+dx_2^2-dx_{3}^2.
$$
 \[\]
 The hyperboloid model of hyperbolic space $\mathbb{H}^2$ is the submanifold

 $$\{x\in \mathbb{R}^{2,1}:x_1^2+x_2^2-x_{3}^2=-1, x_{3}>0\}.$$
  \[\]
  The total geodesic line $U_{x_2}$ along $x_2$ direction through origin of $\mathbb{R}^2$ can be defined as
$U_{x_2}=\{x\in \mathbb{H}^2:\ x_2=0\}$. The general total geodesic line along $x_2$ direction can be generated through hyperbolic rotation. More precisely, define $A^{t}_{x_2}=Id_{\mathbb{R}^{1}}\bigotimes\widetilde{A}_{x_2}^{t} $, where $\widetilde{A}_{x_2}^{t}$ is the hyperbolic rotation on $\mathbb{R}^{1,1}$ and is given by

\begin{align*}
\widetilde{A}_{x_2}^{t}=\bigg{(}
\begin{array}{c}
\cosh t, \ \sinh t\\
\sinh t,\  \cosh t
\end{array} \bigg{)}.
\end{align*}
\[\]
Then $\{U_{x_2}^{t}:=A_{x_2}^{t}(U_{x_2})\}_{t\in \mathbb{R}}$ is a family of total geodesic line along the $x_2$ direction and they are pairwise disjoint and constitutes the whole hyperbolic space $\mathbb{H}^2$. The total geodesic line $U_{\nu}$ along $\nu\in \mathbb{S}^{1}$ direction through origin of $\mathbb{R}^2$ can be defined as $U_{\nu}=\{x\in \mathbb{H}^2:(x_1,x_2)\cdot\nu=0\}$. For any $x'=(x_1,x_2)\in \mathbb{R}^2$, we can write $x'=(x',\nu)\nu+y'$ and $y'$ is orthogonal to $\nu$. $U_{\nu}^{t}=Id_{\nu^{\perp}}\bigotimes\widetilde{A}_{x_2}^{t}$. Simple calculation gives that

$$U_{\nu}^{t}=\left(\sinh t x_{3}\nu+y', \cosh t x_{3}\right), \ \ (x', x_{3})\in U_{\nu}.$$
\[\]
It is easy to check that $U_{\nu}^{t}$ are pairwise disjoint and also constitute the whole hyperbolic space $\mathbb{H}^2$.

\medskip

\medskip

\subsection{Total geodesic line of Poincare disk model and reflection:}
Let $\phi$ be the isometric map from two dimensional Hyperboloid model to Poincare disk. The $\phi$ can be obtained by stereographic projection from the hyperboloid to the plane $\{x_{3}=0\}$ taking the vertex from which to project to be $(0,0,-1)$.
In fact, we can write the map $\phi$ as
$\phi: x\in \mathbb{H}^2\mapsto \frac{x'}{x_{3}+1}\in \mathbb{\mathbb{B}}^2$. Under the map $\phi$, by careful calculation, one can check that
the total geodesic line $U_{\nu}$ of hyperboloid $\mathbb{H}^2$ becomes $\{x \in \mathbb{B}^2:\ x\cdot \nu=0\}$, which is also a geodesic line of Poincare disk. Furthermore $\{\phi(U_{\nu}^t)\}_{t\in \mathbb{R}}$ are pairwise disjoint and constitutes Poincare disk $\mathbb{B}^2$.
 The Mobius transformation(see \eqref{trans-1}) includes rotation and is isometric from $\mathbb{B}^2$ to $\mathbb{B}^2$. It is well known that for any fixed $\nu \in \mathbb{S}^{1}$, $\{T_a\left(\phi(U_{\nu})\right)\}_{a\in \mathbb{B}^2}$ generates all the geodesic line of Poincare disk.
\vskip0.1cm

For the total geodesic line $\phi(U_{x_2})=\{x\in \mathbb{B}^2: x_2=0\}$ along the $x_2$ direction in $\mathbb{B}^2$, it is easy to define the reflection $I_{x_2}$ about $\phi(U_{x_2})$ through $I_{x_2}(x_1, x_2)=(x_1, -x_2)$ for $x=(x_1, x_2)\in \mathbb{B}^2$. Obviously, through geodesic line equation, one can easily verify that $I_{x_2}$ maps the geodesic line to geodesic line. One can similarly define the reflection $I_{\nu}$ about $\phi(U_{\nu})$ through $I_{\nu}(x_1, x_2)=-(x\cdot\nu)\nu+y$, where $x=(x_1, x_2)=(x\cdot\nu)\nu+y$ and $y\in \nu^{\perp}$. Obviously, $I_{\nu}(\phi(U_{\nu}))=\phi(U_{\nu})$.
\vskip0.1cm

Now, we start to define the reflection about the general total geodesic line in $\mathbb{B}^2$. According to the definition of total geodesic line, any total geodesic line can be written as $T_a(\phi(U_{x_2}))$ for some $a>0$. The reflection $I_{x_2}^{a}$ about $T_a(\phi(U_{x_2}))$ can be defined through

$$
I_{x_2}^{a}(x)=T_{a} \circ I_{x_2}\circ T_{a},
$$
\[\]
since $T_a^{-1}=T_a$. Simple calculation gives that $I_{x_2}^{a}\left(T_a(\phi(U_{x_2}))\right)=T_a\circ I_{x_2}\circ T_a\left(T_a(\phi(U_{x_2}))\right)=T_a\circ I_{x_2}(\phi(U_{x_2}))=T_a(\phi(U_{x_2}))$.

\vskip0.1cm

\medskip

\section{the Proof of Theorem \ref{Th1.1th}}\label{s5}
The proof of Theorem \ref{Th1.1th} is divided into two subsections. In the first subsection, we show the radial symmetry of solutions to
the Hardy-type mean field equation \eqref{quan}. In the second subsection, we prove that when $\lambda\rightarrow 0$, the solution $u_\lambda$ must blow up at the origin, and there hold that

$$
\int_{B_{1}}  {\lambda e^{u_\lambda} dx}\rightarrow 8\pi \ \ \ {\rm and}
 \ \ \
u_\lambda(x)\rightarrow  8\pi G( x,0)\ \ \  {\rm in} \ \  C_{loc}^1(B_1 \backslash \{0\}).
$$
\[\]

\subsection{The
symmetry of solutions to the Hardy-type mean field equation}\label{s4}
 We  first consider the Green's function $G(x,y)$ of the operator $-\Delta-\frac{1}{(1-|x|^2)^2}$ with the singularity at $x\in B_1$.
Obviously $G(x,y)$ satisfies equation

\begin{equation}\label{3-1}
\begin{cases}
-\Delta G(x,y)-\frac{G(x,y)}{(1-|x|^2)^2}=\delta_{x}(y)\ {\rm in} \   B_1, \\
\ \ \ \ \ \ \ \ \ \ \ \ \ \   \ \ \ \ G_x(y)=0\ \  \ \ \  \ {\rm on}  \  \partial B_1.
\end{cases}
\end{equation}
\[\]
We can also rewrite equation \eqref{3-1} in Poincare disk $\mathbb{B}^2$. It is not difficult to check that
$G(x,y)$ satisfies the following equation

\begin{equation}
\begin{cases}
-\Delta_{\mathbb{B}^2} G(x,y)-\frac{1}{4}G(x,y)=\delta_{x}(y) \  {\rm in} \  \mathbb{B}^2, \\
\lim\limits_{\rho(y)\rightarrow +\infty}G(x,y)=0.
\end{cases}
\end{equation}
\[\]
From \cite{LuYangQ3}, we know that

$$(-\Delta_{\mathbb{B}^2}-\frac{1}{4})^{-1}=\frac{1}{4}\big(\cosh(\frac{\rho}{2})\big)^{-1}F(\frac{1}{2}; \frac{1}{2}; 1; \cosh^{-2}(\frac{\rho}{2})).$$
\[\]
According to the integral formula of hypergeometric function, we obtain that

\begin{equation}
\begin{split}
(-\Delta_{\mathbb{B}^2}-\frac{1}{4})^{-1}&=\frac{1}{4}\big(\cosh(\frac{\rho}{2})\big)^{-1}F(\frac{1}{2}; \frac{1}{2}; 1; \cosh^{-2}(\frac{\rho}{2}))\\
&=\frac{1}{4\pi}\cosh^{-1}(\frac{\rho}{2})\int_{0}^{1}t^{-\frac{1}{2}}(1-t)^{-\frac{1}{2}}(1-t\cosh^{-2}(\frac{\rho}{2}))^{-\frac{1}{2}}dt\\
&=\frac{1}{4\pi}\int_{0}^{1}t^{-\frac{1}{2}}(1-t)^{-\frac{1}{2}}(\cosh^2(\frac{\rho}{2})-t)^{-\frac{1}{2}}dt\\
&=\frac{1}{4\pi}\int_{0}^{1}t^{-\frac{1}{2}}(1-t)^{-\frac{1}{2}}(\sinh^2(\frac{\rho}{2})+t)^{-\frac{1}{2}}dt.\\
\end{split}
\end{equation}
\[\]
Hence the Green's function $G(x,y)$ with singularity at $x$ can be written as

$$G (x,y)=\frac{1}{4\pi}\int_{0}^{1}t^{-\frac{1}{2}}(1-t)^{-\frac{1}{2}}(\sinh^2(\frac{\rho(x,y)}{2})+t)^{-\frac{1}{2}}dt,$$
\[\]
where $\rho(x,y)=\rho(T_{x}y,0)$.
\medskip

 \medskip

\begin{lemma}\label{lem8.1}(Lemma 4.1, \cite{LuYangQ3})
For any complex number $s=\sigma+i\tau$ with $\sigma>0$ and $u>0$, define

 $$
\varphi(s,u)=\frac{1}{4\pi}\int_{0}^{1}t^{s-1}(1-t)^{s-1}(t+u)^{-s}dt.
$$
\[\]
Then $\varphi(s,u)$ is analytic in $s$, $C^{\infty}$ in $u$. Furthermore, for any fixed $s$, there holds
\vskip0.1cm

(1) $\varphi(s,u)=-\frac{1}{4\pi}\log(u)+O(1)$, \ \ if $u\rightarrow 0$;
\vskip0.1cm

(2) $\varphi(s,u)=O(u^{-\sigma})$,\ \ if $u\rightarrow +\infty$.
\end{lemma}
Applying Lemma \ref{lem8.1}, we can derive that

\begin{equation}\label{green estimate1}
\begin{split}
G(x,y)&=\varphi(\frac{1}{2},\sinh^2(\frac{\rho(x,y)}{2}))\\
&= -\frac{1}{2\pi}\log(\sinh(\frac{\rho(x,y)}{2}))+C\\
&=-\frac{1}{2\pi}\log \big(\frac{|x-y|}{(1-|x|^2)^{\frac{1}{2}}(1-|y|^2)^{\frac{1}{2}}}\big)+C\\
&=-\frac{1}{2\pi}\log |x-y|+\frac{1}{4\pi}\log (1-|x|^2)+\frac{1}{4\pi}\log (1-|y|^2)+C\\
&\leq -\frac{1}{2\pi}\log |x-y|+C,
\end{split}
\end{equation}
\[\]
where we use the fact

$$
\sinh(\frac{\rho(x,y)}{2})=\frac{|x-y|}{(1-|x|^2)^{\frac{1}{2}}(1-|y|^2)^{\frac{1}{2}}}.
$$
\medskip

\medskip

Now, we are in a position to show the symmetry of solutions to the Hardy-type mean field equation.

\medskip

\begin{proposition}\label{pro1}
	Assume that
$u$ satisfies equation \eqref{quan}.
Then $u$ is positive and radially decreasing about the origin.
\end{proposition}

We rewrite the equation \eqref{quan} as

\begin{equation}
\Bigg\{ {\begin{array}{*{20}{c}}
  \medskip
   { - \Delta_{\mathbb{B}^2}  u-\frac{1}{4} u  = \lambda e^u \frac{(1-|x|^2)^2}{4}\ \ \ \ {\rm in} \ \mathbb{B}^2 ,}  \\
     \medskip
 { u\in \mathcal{H}}.
\end{array}} \Bigg.
\end{equation}
\begin{proof}
 By Green's representation theorem, we know that

 \begin{equation}\label{3-7}
 u(x)=\int_{\mathbb{B}^2}G(x,y)\lambda e^u  \frac{(1-|x|^2)^2}{4} dV_{\mathbb{B}^2},
 \end{equation}
 \[\]
 where $G(x,y)$ is the Green's function of the operator $-\Delta_{\mathbb{B}^2}-\frac{1}{4}$ on the Poincare disk $\mathbb{B}^2$. Non-negativity of Green's function directly implies the solution $u$ being positive. Hence, we only need to show the radial symmetry of positive solutions to \eqref{3-7}. We adapt the moving plane method in integral forms of hyperbolic space developed in \cite{LiLuYang}. To perform the moving plane arguments, we fix one specific direction and for such direction, consider $t>0$, $\phi(U_{x_2}^t)$ splits the Poincare disk $\mathbb{B}^2$ into two parts. We denote by $\Sigma_{t}=\bigcup_{s>t}\phi(U_{x_2}^s)$ and $u_{t}(x)=u(I^t_{x_2}(x))$.
 Direct calculations give that for any $x\in \Sigma_{t}$, there hold

 \begin{equation}\begin{split}
 u(x)&=\int_{\Sigma_{t}}G(x,y)\lambda e^u  \frac{(1-|y|^2)^2}{4}dV_{\mathbb{B}^2}+\int_{\mathbb{B}^2\setminus \Sigma_{t}}G(x,y)\lambda e^u  \frac{(1-|y|^2)^2}{4}dV_{\mathbb{B}^2}\\
 &=\int_{\Sigma_{t}}G(x,y)\lambda e^u  \frac{(1-|y|^2)^2}{4}dV_{\mathbb{B}^2}+\int_{\Sigma_{t}}G(x, I^t_{x_2}(y))\lambda e^{u_{t}}  \frac{(1-|I_{x_2}^t(y)|^2)^2}{4}dV_{\mathbb{B}^2}
 \end{split}\end{equation}
 and
 \begin{equation}\begin{split}
 u_t(x)&=\int_{\Sigma_{t}}G(I_{x_2}^t(x),y)\lambda e^u \frac{(1-|y|^2)^2}{4}dV_{\mathbb{B}^2}+\int_{\mathbb{B}^2\setminus \Sigma_{t}}G(I_{x_2}^t(x), y)\lambda e^u \frac{(1-|y|^2)^2}{4}dV_{\mathbb{B}^2}\\
 &=\int_{\Sigma_{t}}G(I_{x_2}^t(x),y) \lambda e^u \frac{(1-|y|^2)^2}{4}dV_{\mathbb{B}^2}+\int_{\Sigma_{t}}G(I_{x_2}^t(x), I_{x_2}^t(y))\lambda e^{u_t}  \frac{(1-|I_{x_2}^t(y)|^2)^2}{4}dV_{\mathbb{B}^2}.
 \end{split}\end{equation}
 \[\]
Since $I_{x_2}^t$ is an isometry, we have $G(I_{x_2}^t(x),y)=G(x, I_{x_2}^t(y))$ and $G(I_{x_2}^t(x), I_{x_2}^t(y))=G(x,y)$. Hence we can derive that

\begin{equation}
\begin{split}
\label{identity}
u(x)-u_t(x)&=\int_{\Sigma_{t}}\big(G(I_{x_2}^t(x),y)-G(x,y)\big)\Big(\frac{(1-|I_{x_2}^t(y)|^2)^2}{4}-\frac{(1-|y|^2)^2}{4}\Big)\lambda e^{u}dV_{\mathbb{B}^2}\\
& \ \ +\int_{\Sigma_{t}}\left(G(I_{x_2}^t(x),y)-G(x,y)\right)
\frac{(1-|I_{x_2}^t(y)|^2)^2}{4}\lambda\big(e^{u_t}-e^{u}\big)dV_{\mathbb{B}^2}.
\end{split}
\end{equation}
\medskip

\emph{\textbf{Step 1}}. We compare the values of $u_t(x)$ and $u(x)$. We first show that for $t$ sufficiently negative, there holds

\begin{equation}
\label{starting}u_t(x)\leq u(x),\ \ \forall x \in \Sigma_{t}.
\end{equation}
Define

$$
\Sigma_{t}^{u}=\{x\in \Sigma_{t}:\ u(x)<u_t(x)\}.
$$
\[\]
We will show that for $t$ sufficiently negative, $\Sigma_{t}^{u}$ must be empty.
 By \eqref{identity} and the mean-value theorem, we derive
\begin{equation}
\begin{split}\notag\\
u_{t}(x)-u(x)&\leq \int_{\Sigma_{t}}\left(G(x,y)-G(I_{x_2}^t(x),y)\right)\frac{(1-|I_{x_2}^t(y)|^2)^2}{4}\lambda\big(e^{u_t}-e^{u}\big)dV_{\mathbb{B}^2}\\
&\leq \int_{\Sigma_{t}^u}\left(G(x,y)-G(I_{x_2}^t(x),y)\right)\frac{(1-|I_{x_2}^t(y)|^2)^2}{4}\lambda\big(e^{u_t}-e^{u}\big)dV_{\mathbb{B}^2}\\
&\leq \int_{\Sigma_{t}^u}G(x,y)\frac{(1-|I_{x_2}^t(y)|^2)^2}{4}\lambda e^{u_t}(u_t-u)dV_{\mathbb{B}^2}.
\end{split}
\end{equation}
\[\]
Through Hardy-Littlewood-Sobolev inequality in Poincare disk (see subsection 2.2), we derive that

\begin{equation}\begin{split}
\|u_{t}-u\|_{L^q(\Sigma_{t}^{u})}&\leq C_{q}^{-1}\|\ \frac{(1-|I_{x_2}^t(y)|^2)^2}{4}\lambda e^{u_t}(u_t-u)\|_{L^{q'}(\Sigma_{t}^{u})}\\
&\lesssim \|\frac{(1-|I_{x_2}^t(y)|^2)^2}{4}\lambda e^{u_t}\|_{L^{r}(\Sigma_{t}^{u})}\|u_t-u\|_{L^q(\Sigma_{t}^{u})},
\end{split}\end{equation}
\[\]
where $2<q<+\infty$ and $\frac{1}{q'}=\frac{1}{q}+\frac{1}{r}$. By Hardy-Moser-Trudinger inequality \eqref{HMT}, we know that

\begin{equation}\begin{split}
\int_{\mathbb{B}^2}\Big(\frac{(1-|I^{t}_{x_{2}}(y)|^2)^2}{4}e^{u_{t}}\Big)^rdV_{\mathbb{B}^2}\leq \int_{B_1}e^{r u_{t}}dx<+\infty.
\end{split}\end{equation}
\[\]
Hence, $\|\frac{(1-|I_{x_2}^t(y)|^2)^2}{4}\lambda e^{u_t}\|_{L^{r}(\Sigma_{t}^{u})}$ is sufficiently small for sufficiently negative $t$. This implies that $\Sigma_{t}^{u}$ must be empty for sufficiently negative $t$. Then we accomplish the proof of Step 1.
\medskip

\emph{\textbf{Step 2}}. Inequality \eqref{starting} provides a starting point to move the plane $\phi(U_{x_2}^t)$. Define

$$t_0=\sup \{\ t:\ u_s(x)\leq u(x), \ s\leq t,\ \ \forall x\in \Sigma_{s}\}.$$
\[\]
We show that $t_0\geq 0$. Suppose on the contrary that $t_0<0$, we only need to prove that the plane can be moved further to the right, which is a contradiction with the definition of $t_0$.  Obviously,

$$u(x)\geq u_{t_0}(x),\ \ \forall x\in \Sigma_{t_0}.$$
\[\]
  By \eqref{identity}, we see that

   $$u(x)>u_{t_0}(x),\ \ \forall x\in \Sigma_{t_0}.$$
 \[\]
 Next, we will show that there exists an $\epsilon>0$ such that for any $t\in [t_0, t_0+\epsilon)$, there holds

$$u(x)> u_{t}(x),\ \ \forall x\in \Sigma_{t}.$$
\[\]
Let $$\overline{\Sigma_{t_0}^{u}}=\{x\in \Sigma_{t_0}\ |\ u(x)\leq  u_{t_0}(x)\}.$$
\[\]
Obviously, there hold that $\overline{\Sigma_{t_0}^{u}}$ has the measure zero and $\lim\limits_{t\rightarrow t_0}\Sigma_{t}^{u}\subseteq \overline{\Sigma_{t_0}^{u}}$. Then it follows that there exists $\epsilon>0$ such that for any $t\in [t_0, t_0+\epsilon)$, the integral $\|\frac{(1-|I_{x_2}^t(y)|^2)^2}{4}\lambda e^{u_t}\|_{L^{r}(\Sigma_{t}^{u})}$ is sufficiently small. Recall that

 $$
\|u_t-u\|_{L^q(\Sigma_{t}^{u})}\lesssim  \Big\|\frac{(1-|I_{x_2}^t(y)|^2)^2}{4}\lambda e^{u_t}\Big\|_{L^{r}(\Sigma_{t}^{u})}\|u_t-u\|_{L^q(\Sigma_{t}^{u})}.
$$
\[\]
This deduces that for any $t\in [t_0, t_0+\epsilon)$, there holds $u(x)\geq u_t(x),\ \ \forall x\in \Sigma_{t}$. Then we accomplish the proof of Step 2.
\medskip

\emph{\textbf{Step 3}}. Since $t_0\geq 0$, we derive that $u(x)\geq u_0(x),\ \forall x\in \Sigma_{0}$. Observing $\Sigma_{0}=\{(x_1, x_2)\in \mathbb{B}^2, x_2\geq 0\}$, hence we derive that for any $x\in \mathbb{B}^2$ with $x_2\geq 0$, $u(x_1, x_2)\geq u(x_1, -x_2)$. Similarly, we can move the plane from the $t=+\infty$ to derive that $u(x)\leq u_0(x),\ \forall x\in \Sigma_{0}$. To sum up, we conclude that for any $x\in \mathbb{B}^2$,
there holds $u(x_1,x_2)=u(x_1,-x_2)$. Since the direction can be chosen arbitrarily, we conclude that $u(x)$ is radially symmetric and strictly decreasing about the origin.
\end{proof}

\subsection{Quantitative properties for the Hardy-type mean field equation}
In this subsection, we will establish the quantitative properties for solutions of the Hardy-type mean field equation (\ref{quan}).
  The proof will be divided into two steps. In step 1, we will show that the solution $u_\lambda$ of equation (\ref{quan}) must blow up at the origin $0$ as $\lambda\rightarrow 0$. In step 2, we further prove that

$$\lim\limits_{\lambda\rightarrow 0}\int_{B_1}  {\lambda e^{u_\lambda} dx}= 8\pi \ \ \
{\rm and} \ \ \
\lim\limits_{\lambda\rightarrow 0}u_\lambda(x)=  8\pi {G( x,0)} \ \ \ {\rm  in} \ \ \ C^1_{loc}(B_1\backslash\{0\}).$$
\medskip

\emph{\textbf{Step 1}}. we will show that $u_\lambda$ just blows up at the origin $0$ when $\lambda$ approaches to zero.
 We first present a similar Brezis-Merle Lemma.
\begin{lemma}
\label{B-M}
Let $u$ be a $C^2$ solution of

\begin{equation}
\left\{ {\begin{array}{*{20}{c}}
   { (- \Delta  - \frac{1}{(1-|x|^2)^2})u = f(u)}  & {{\rm in} \ B_1} , \\
   {u = 0} & {\ \ {\rm on} \ \partial B_1},  \\
\end{array}} \right.
\label{3.1}
\end{equation}
\[\]
where $f\in L^1({B_1} )$ and $f\geq 0$. Then for any $\delta \in(0,4\pi)$, there will hold

\[
\int_{B_1}  {\exp \left\{ {\frac{\left( {4\pi  - \delta } \right)\left|u\right|}{{\left\| f \right\|}_{{L^1}\left({B_1} \right)}}} \right\}dx} \leq C.
\]
\end{lemma}

\begin{proof}
Using Green's  representation theorem  and the monotonicity of $G(x,y)$(see \eqref{green estimate1}), for any $x\in B_1$,

\[
u\left( x \right)
=
\int_{B_1}  G(x,y)f(u(y))dy
\leq
- \frac{1}{2\pi }\int_{B_1} \log \left| {x - y} \right|f(u(y))dy
+C\int_{B_1}f(u(y))dy.
 \]
\[\]
By Jensen's inequality, we deduce that

\begin{align*}
 \int_{B_1}   {\exp \Big\{ {( {4\pi-\delta } )\frac{\left|u\right|}{{{{\left\| f \right\|}_{{L^1}\left({B_1} \right)}}}}} \Big\}dx}&
 \leq
 \int_{B_1}   {\exp \Big\{ {\left( {4}\pi  - \delta  \right)\int_{B_1}  \frac{(\log \left| {x - y} \right|^{ - \frac{1}{2\pi }}+C) f(u(y))}{\|f\|_{{L^1}(B_1)}} dy}  \Big\}dx}  \\
  &\leq C\int_{B_1}  {\int_{B_1} {\exp \left\{ {\left( {4}\pi  - \delta  \right)\big(\log \left| {x - y} \right|^{ - \frac{1}{2\pi }} }  +C\big)\right\}} dydx} \\
 &
  =e^{C}\int_{B_1}  {\int_{B_1}  | x - y|^{ - \frac{4\pi  - \delta}{2\pi }}dydx}
  \leq C .
 \end{align*}
\end{proof}

Using Lemma \ref{B-M}, one can derive that $u_\lambda \in L^p(B_1)$ for any $p\geq 1$, i.e.,

$$\int_{B_1} u_\lambda (x)dx \leq C.$$
\[\]
Combining the radial lemma, for any $r\leq 1$, there holds that

$$u_\lambda (r)r^2\leq C \int_{B_1} u_\lambda (x)dx,$$
\[\]
which implies that $u_\lambda\leq \frac{C}{r^2}$ for some constant C. Then for any $x\in B_1\backslash \{0\}$, one can obtain

$$\|u_\lambda\|_{L^\infty(B_1\backslash \{0\})}\leq C.$$
\[\]
Indeed, $u_\lambda(0)$ is unbounded when $\lambda$ approaches to zero. If not, there
exists some positive constant $C$ such that $\|u_\lambda\|_{L^\infty(B_1)}\leq C$.
One can easily conclude that

\[\lim\limits_{\lambda\rightarrow 0}\lambda\int_{B_1} {e^{u_\lambda}} dx\leq \lim\limits_{\lambda\rightarrow 0}\lambda e^{\|u_\lambda\|_{L^\infty(B_1)}}|B_1|= 0,
\]
\[\]
which contradicts with the assumption, $\lambda\int_{B_1}{e^{u_\lambda}} dx\geq C_1 >0$, from Theorem \ref{Th1.1th}.
Thus, $u_\lambda$  must have and only have one blow-up point $0$ when $\lambda$ approaches to zero.
\medskip

\medskip

\emph{\textbf{Step 2}}. we will show that
$u_\lambda(x)\rightarrow  u_0(x)=8\pi G( x,0)$ in $C_{loc}^1(B_1 \backslash \{0\})$ as $\lambda\rightarrow 0$.
\medskip

Defining $\mu_\lambda:=\lambda e^{u_\lambda}dx$, then $\mu_\lambda(B_1)=\int_{B_1}\lambda e^{u_\lambda}dx\leq C$. Hence, there exists a $\mu_0\in\mathfrak{M}(B_1)$, the set of all real bounded Borel measures on $B_{1},$ such that $\mu_\lambda\rightharpoonup \mu_0$ in the sense of measure.
By the previous estimate about $\|u_\lambda\|_{L^\infty(B_1 \backslash \{0\})},$
we know that

\[
\lim\limits_{\lambda\rightarrow 0} \mu_\lambda(B_1\backslash \{0\})= \lim\limits_{\lambda\rightarrow 0}\int_{B_1\backslash \{0\}}\lambda e^{u_\lambda}dx= 0.
\]
\[\]
Then it implies that

$$
\mu_\lambda\rightharpoonup \mu_0= \mu_0(0)\delta_{0}.
$$
\medskip

Next, we claim that
\begin{lemma}
\label{8th}
$u_\lambda(x)\rightarrow \mu_0(0)G(x,0)$  in $C_{loc}^1(B_1\backslash \{0\})$ as $\lambda\rightarrow 0$.
\end{lemma}
\medskip

\begin{proof}
Since $u_\lambda\in W^{1,2}_0(B_1)$ satisfies the equation

 \begin{equation}
\label{2.6-1}
{- \Delta u_\lambda- \frac{1}{(1-|x|^2)^2} u_\lambda = \lambda e^{u_\lambda} \in L^1(B_1)},
\end{equation}
\[\]
testing equation \eqref{2.6-1} with $u_\lambda^t:=\min\{u_\lambda,t\}$,
we obtain that

\[ \int_{B_1}  {|\nabla u_\lambda^t|^2 dx}-\int_{B_1}\frac{|u_\lambda^t|^2}{(1-|x|^2)^2}dx
   =\int_{B_1}\lambda e^{u_\lambda^t}u_\lambda^t dx \leq C_2 t.
  \]
  \[\]
 Let $u_\lambda^*$ be the classical rearrangement of $u_\lambda^t$ and $|B_\rho(x)|=|\{x\in B_r(x):u_\lambda^*\geq t\}|$, where $B_r=\{x\in \mathbb{R}^2:\|x\|\leq r\}$. According to the classical rearrangement, we have that

 \begin{equation}
\label{2.6}
\mathop {\inf }\limits_{\phi\in W^{1,2}_0(B_r),\ \phi|_{B_\rho}=t}\int_{B_r} {|\nabla \phi|^2}dx  \leq \int_{B_r} {|\nabla u_\lambda^*|^2}dx \leq C_2 t+\int_{B_1}\frac{|u_\lambda^t|^2}{(1-|x|^2)^2}.
\end{equation}
\[\]
We have known that $ u_\lambda\in W^{1,2}_0(B_1)$, which implies that $\int_{B_1}\frac{|u_\lambda^t|^2}{(1-|x|^2)^2}\leq C
$. Thus, the right hand of \eqref{2.6} can be written as $ C$.
  It is not difficult to check that (see \cite{Li}) the infimum on the left-hand side of \eqref{2.6} is attained by

  \[ \varphi _1  ( x)=\left\{ {\begin{array}{*{20}c}
   \medskip

   {t \log  \frac{r} {|x|}/{\log \frac{r} {\rho}}, \ \ \  {\rm in}\ B_r\backslash {B_\rho}},  \\
   {t,\ \ \ \ \ \ \  \ \  \ \   {\rm in}\ {B_\rho}.}
 \end{array} } \right.\]
 \[\]
 Calculating $\|\nabla \varphi_1\|_2^2$, by \eqref{2.6}, we get $\rho \leq r e^{-C t}$.
 Thus, there holds

$$
|\{x\in B_1:u_\lambda \geq t\}|=|B_\rho|\leq \pi r^2 e^{-2C t}.
$$
\[\]
Using Taylor's expansion formula, for any $0<v<2C $, we have

\[\int_{B_1}  {e^{v u_\lambda} dx}\leq e^v  +\sum\limits_{m = 1}^\infty e^{v (m+1)}|\{x\in B_1: m\leq u_\lambda\leq m+1\}|
    \leq C ,\]
\[\]
which implies that $u_\lambda$ is uniformly bounded in $L^2(B_1)$.
\medskip

Testing equation \eqref{2.6-1} with $\log \frac{1+2u_\lambda}{1+u_\lambda}$ and applying Young's inequality,
for any $1<q<2$, one can obtain   that

\begin{align*}
   \int_{B_1}  {|\nabla u_\lambda|^q dx}&\leq\int_{B_1} \frac {|\nabla u_\lambda|^2 }{(1+u_\lambda)(1+2u_\lambda)}dx
   +C_{q}\int_{B_{1}}( (1+u_\lambda)(1+2u_\lambda))^\frac{q}{n-q}dx
   \\&\leq \int_{B_1}\lambda e^{u_\lambda}\log \frac{1+2u_\lambda}{1+u_\lambda}+\int_{B_1}\frac{|u_\lambda^t|^2}{(1-|x|^2)^2}\log \frac{1+2u_\lambda}{1+u_\lambda}+C \int_{B_1}  {e^{v u_\lambda} dx}
    \\
    &\leq
    C\log 2+ C\log 2 +C=C.
\end{align*}
\[\]
Namely,
 $u_\lambda$ is uniformly bounded in $W_0^{1,q}(B_1)$ for any $1<q<2$. Then, there exists $u_0\in W^{1,q}_0(B_1)$ such that $u_\lambda \rightharpoonup u_0$ in $W^{1,q}_0(B_1)$, where $u_0$ satisfies

$$- \Delta  u_0- \frac{1}{(1-|x|^2)^2} u_0 =  \mu_0(0)\delta_{0}.$$
\[\]
Similar to Lemma 10 in  \cite{yang}, we derive that

$$ u_0 =\mu_0(0)G(x,0).$$
\[\]
 Applying the regularity estimate for quasilinear differential operator (see \cite{Li}),
we deduce that $u_\lambda\rightarrow u_0$ in $C_{loc}^1(B_1\backslash \{0\})$.
\end{proof}

Thus, in order to prove $u_\lambda(x)\rightarrow  8\pi G(x,0)$ in $C_{loc}^1(
B_1\backslash \{0\})$ as $\lambda\rightarrow 0$,
we just need to verify that
$\mu_0(0)=8\pi$.
First, we need the local Pohozaev identity for the Hardy-type mean field equation (\ref{quan}).
\medskip

\begin{lemma}
\label{Poho1}
For any $r\leq 1$, there holds

\begin{equation}
\label{Poho}
\begin{split}
 4{\int_{B_r}  {F(u_\lambda)dx} } & = {\int_{\partial{B_r}}\left| {\nabla u_\lambda} \right|^2\langle x,\nu\rangle dS}  -\frac{1}{2}
\int_{B_r}u_\lambda^2
\ \Delta (\frac{1}{1-|x|^2}) dx
\\
&+\frac{1}{2}\int_{\partial B_r}u_\lambda^2 {\frac{\partial (\frac{1}{1-|x |^2})}{\partial n}dS}+\int_{\partial B_r}F(u_\lambda)\frac{{\partial (|x|^2)}}{{\partial n}}dS, \\
\end{split}
\end{equation}
where $F(u_\lambda)=\lambda e^{u_\lambda}.$
\end{lemma}
\begin{proof}
Multiplying equation (\ref{quan}) by $\langle x, \nabla u_\lambda\rangle$ and integrating over ${B_r}$, we get

\[
\int_{B_r}  {\Big(- \Delta- \frac{1}{(1-|x|^2)^2} \Big) u_\lambda \cdot\langle x, \nabla u_\lambda\rangle dx} = \int_{B_r}\lambda e^{u_\lambda}\cdot \langle x,\nabla u_\lambda\rangle dx.\]
\[\]
We rewrite this expression as
\[A+B=C.\]
\[\]
Via the divergence theorem and direct computation, the term on the left is

\[A:=\int_{B_r}  {- \Delta u_\lambda  \cdot
\langle  x ,\nabla u_{\lambda}\rangle dx} =-\frac{1}{2}\int_{\partial B_r}\left| {\nabla u_\lambda} \right|^2\langle x,\nu \rangle
dS\]
\begin{equation*}
and
\begin{split}
B:&=-\int_{B_r}\frac{u_\lambda}{(1-|x|^2)^2}\cdot \langle
x , \nabla u_{\lambda}\rangle  dx
=-\frac{1}{4}
\int_{B_r}\nabla (\frac{1}{1-|x|^2})\cdot \nabla ( u_\lambda^2)dx
\\&
=\frac{1}{4}
\int_{B_r}u_\lambda^2
\ \Delta (\frac{1}{1-|x|^2}) dx
-\frac{1}{4}\int_{\partial B_r}u_\lambda^2 {\frac{\partial (\frac{1}{1-|x |^2})}{\partial n}dS} .
\end{split}
\end{equation*}
\[\]
And
 the term on the right-hand side is
 \[
C:=\int_{B_r}\lambda  e^{u_\lambda}\cdot \langle
x,\nabla u_\lambda\rangle
dx=\frac{1}{2}\int_{\partial B_r}F(u_\lambda)\frac{{\partial (|x |^2)}}{\partial n}dS-2\int_{ B_r}F(u_\lambda)dx.\]
Combining above estimates, we accomplish the proof of Lemma \ref{Poho1}.
\end{proof}
\medskip

\begin{lemma}
\label{9th}
$\mu_0(0)=8\pi$.
\end{lemma}
\medskip

\begin{proof}

Since $u_\lambda$ converges strongly to $\mu_0(0)G(x,0)$ in $C_{loc}^1(B_r\backslash \{0\})$
as $\lambda\rightarrow 0$, we derive that
\[ \medskip
    \lim\limits_{r\rightarrow 0}  \lim\limits_{\lambda\rightarrow 0} \int_{\partial B_r}F(u_\lambda)\frac{{\partial (|x|^2)}}{{\partial n}}dS = 0,\,\, \ \ \
 \lim\limits_{r\rightarrow 0}  \lim\limits_{\lambda\rightarrow 0} \int_{\partial B_r}\left| {\nabla u_\lambda} \right|^2 \langle
 x,\nu\rangle  dS=  \frac{1}{2\pi}\mu_0^2(0),
\]
and
\[\medskip
 \lim\limits_{r\rightarrow 0}  \lim\limits_{\lambda\rightarrow 0} \int_{\partial B_r}u_\lambda^2 \frac{\partial (\frac{1}{1-|x |^2})}{\partial n}dS =0,\,\,\ \ \ \ \ \ \ \ \ \ \ \ \
 \lim\limits_{r\rightarrow 0}  \lim\limits_{\lambda\rightarrow 0} \int_{B_r}u_\lambda^2
\ \Delta (\frac{1}{1-|x|^2}) dx=0 .
\]
\[\]
Combining above estimates, one can obtain

\[  \lim\limits_{r\rightarrow 0}  \lim\limits_{\lambda\rightarrow 0} 4\int_{B_r}  {F(u_\lambda)dx}  = \frac{1}{2\pi}\mu_0^2(0) .\]
\[\]
This together with
\[  \lim\limits_{r\rightarrow 0}  \lim\limits_{\lambda\rightarrow 0}  \int_{B_r}  {F(u_\lambda)dx}=  \lim\limits_{r\rightarrow 0}  \lim\limits_{\lambda\rightarrow 0}   \int_{B_r}   {\lambda e^{u_\lambda}dx}=\mu_0(0),  \]
\[\]
yields that
\[\mu_0(0)=8\pi.\]
\end{proof}
\medskip

\section{the asymptotic behavior of $u_\lambda$ around the origin}\label{s6}
In this section, we will analyze the asymptotic behavior of the solution $u_\lambda$  around the blow-up point.
Assuming that $c_\lambda:=u_\lambda(0)=\max\limits_{B_1}u_\lambda(x)$
 and setting $\lambda r_\lambda^{2}e^{c_\lambda}=8$,  we first prove that
 \begin{lemma}
 \label{r=0}
    $r_\lambda^{2}\rightarrow 0$ as  $\lambda\rightarrow 0$.
 \end{lemma}
\begin{proof}
    We argue this by contradiction. If $\lim\limits_{\lambda\rightarrow 0}r_\lambda^{2}>0$, one can easily get that $\lambda e^{c_\lambda}$ is bounded. Hence the  solution $u_\lambda$ satisfies that

\begin{equation}\begin{cases}
-\Delta u_{\lambda}-\frac{u_\lambda}{(1-|x|^2)^2}=\lambda e^{u_\lambda}:=f_{\lambda},\\
\ \ \ \ \ \ \ \ \ \ \    \   u_\lambda\in \mathcal{H},
\end{cases}\end{equation}
\[\]
where $f_{\lambda}\in L^{\infty}(B_1)$. From Proposition 1 of \cite{Wang-Ye}, we deduce that $u_\lambda$ can be written as
$u_\lambda=u^{1}_{\lambda}+u^{2}_{\lambda}$ with

$$\Big(\int_{B_1}\big(|\nabla u_{\lambda}^2|^2-\frac{u_\lambda^2}{(1-|x|^2)^2}\big)dx\Big)^{\frac{1}{2}}+\big(\int_{B_{\frac{1}{4}}}|\nabla u_\lambda^1|^pdx\big)^{\frac{1}{p}}\leq C_p,$$
\[\]
where $u_{\lambda}^{1}\in W^{1,p}_{0}(B_{\frac{1}{4}})$ and $u_{\lambda}^{2}\in \mathcal{H}$ for any $p\geq 1$. This together with Hardy-Trudinger-Moser inequality gives that
$$\frac{u_\lambda}{(1-|x|^2)^2}+f_{\lambda}\in L^{2}(B_{\frac{1}{4}}).$$
\[\]Then it follows from the standard elliptic estimate that
$u_\lambda \in W^{2,2}(B_{\frac{1}{8}})$, which yields that $u_{\lambda}(0)$ is bounded. This arrives at a contradiction with the assumption $\lim\limits_{\lambda\rightarrow 0}c_\lambda=+\infty$. Thus, we obtain $r_\lambda^2\rightarrow 0$.
\end{proof}
\medskip

\medskip

Letting $v_{\lambda}(x):=r^{2}_{\lambda}u_{\lambda}(x),$
then $v_{\lambda}(x)$ satisfies

\begin{equation*}
\begin{cases}
-\Delta v_{\lambda}-\frac{v_{\lambda}}{(1-|x|^{2})^{2}}
=\lambda r_{\lambda}^{2} e^{u_{\lambda}(x)}:=g_{\lambda}(x)\ \ \ \  \ ~\mbox{in}\ \   ~B_{1},  \\[1mm]
\ \ \ \ \ \ \ \ \ \ \ \ \ \ \ \ \ \ \ \ \  \ \ \ \   \ v_{\lambda}(x)=0  \ \ \ \  \  \ \     \ \  \ \ \ \ \ \   ~\mbox{on}~\partial B_{1},
\end{cases}
\end{equation*}
\[\]
where $g_{\lambda}(x)\in L^{\infty}(B_{1})$.
By Lemma \ref{r=0} we derive that

$$
\lim_{\lambda\rightarrow 0}g_{\lambda}(x)=
\begin{cases}
0,\,\,x\neq 0\\[1mm]
8,\,\,x=0.
\end{cases}
$$
\[\]
Then this deduces the following lemma.

\begin{lemma}\label{lem-10-7-3}
There holds that $r_{\lambda}^{2}c_{\lambda}\rightarrow 0$
as $\lambda\rightarrow 0.$
\end{lemma}
\begin{proof}

By the Green's representation theorem and \eqref{green estimate1}, we have
\begin{equation}\nonumber
\begin{aligned}
v_{\lambda}(x)&=\int_{B_{1}}G(x,y)g_{\lambda}(y)dy\\
&\leq \int_{B_{1}}\Big(-\frac{1}{2\pi}\log|x-y|+C\Big)g_{\lambda}(y)dy,
\end{aligned}
\end{equation}
\[\]
which implies that
\begin{equation}\nonumber
\begin{aligned}
\lim_{\lambda\rightarrow 0}v_{\lambda}(0)
&\leq \lim_{\lambda\rightarrow 0}\int_{B_{1}}\Big(-\frac{1}{2\pi}\log|y|+C\Big)g_{\lambda}(y)dy\\
&=-\lim_{\lambda\rightarrow 0}\lim_{\epsilon\rightarrow 0}\Big(\int_{B_{\epsilon}}+\int_{B_{1}\setminus B_{\epsilon} }\Big)\frac{1}{2\pi}\log|y|g_{\lambda}(y)dy
+C\lim_{\lambda\rightarrow 0}\int_{B_{1}}g_{\lambda}(y)dy\\
&=0.
\end{aligned}
\end{equation}

\end{proof}
Next, setting $\eta_{\lambda}(x):=u_{\lambda}(r_{\lambda}x)-c_{\lambda}$, then $\eta_{\lambda}(x)$ satisfies

\begin{equation}\label{10-7-3}
-\Delta (\eta_{\lambda}(x))=
r^{2}_{\lambda}\frac{\eta_{\lambda}(x)+c_{\lambda}}{(1-|r_{\lambda}x|^{2})^{2}}
+8e^{\eta_{\lambda}(x)},\ \ x\in B_{\frac{1}{r_\lambda}}.
\end{equation}
 Obviously, we conclude the following lemma.

\medskip

\begin{lemma}\label{lem-10-7-1}
 There holds that
\begin{equation}\label{10-7-1}
\eta_{\lambda}(x)
\rightarrow \eta_{0}(x):=-2\log(1+|x|^{2})\,\,\,\text{in}\,\,C^{1}_{loc}(\R^{2})
\end{equation}
and
\begin{equation}\label{10-7-2}
\lim_{R\rightarrow \infty}\lim_{\lambda\rightarrow 0}
\int_{B_{R r_{\lambda}}}\lambda e^{u_{\lambda}}dx
=\int_{\R^{2}}8e^{\eta_{0}}dx=8\pi.
\end{equation}
\end{lemma}

\begin{proof}
Observing the definition of $\eta_\lambda(x)$, we can easily find
that $\eta_{\lambda}(x)\leq 0,\Delta \eta_{\lambda}$
is locally bounded and $\eta_{\lambda}(0)=0$.
From Lemma \ref{lem-10-7-3}, we have known that $r^{2}_{\lambda}\frac{\eta_{\lambda}(x)+c_{\lambda}}{(1-|r_{\lambda}x|^{2})^{2}}\rightarrow 0$.
Then it follows from Harnack inequality that $\eta_{\lambda}(x)\rightarrow \eta_0$
in $C^{1}_{loc}(\R^{2})$, where

\begin{equation}\label{10-7-4}
-\Delta \eta_{0}=8e^{\eta_{0}}\,\,\,\text{in}\,\,\R^{2},\,\,\,\eta_{0}(0)=0.
\end{equation}
Hence, by the uniqueness of solutions to the Cauchy problem \eqref{10-7-4}, we obtain that $\eta_{0}=-2\log(1+|x|^{2})$. Consequently,
\begin{equation*}
   \lim_{R\rightarrow \infty}\lim_{\lambda\rightarrow 0}
\int_{B_{R r_{\lambda}}}\lambda e^{u_{\lambda}}dx
=8\int_{\R^{2}}e^{\eta_{0}}dx=8\pi.
 \end{equation*}
\end{proof}

\medskip

\section{The uniqueness of solutions}\label{s7}
Since it follows from Proposition \ref{pro1} that $u_{\lambda}(y)$ is radial, in order to prove Theorem \ref{unique},
by Cauchy-initial uniqueness for ODE,
we only need to prove
the following result.
\begin{proposition}\label{thuniq}
	Let $u_\lambda^{(1)}$ and $u_\lambda^{(2)}$ be two solutions to problem \eqref{quan}
	which concentrate at the same point $0.$
	Then  there exists  $\lambda_0>0$ such that

	\[u_\lambda^{(1)}(x)\equiv u_\lambda^{(2)}(x),   ~\mbox{for any}~~ \lambda\in (0,\lambda_0),   x\in B_{\delta r_{\lambda}}(0),\]
\[\]
	where $\delta>0$ is a small fixed constant.
\end{proposition}
\medskip

To prove Proposition \ref{thuniq}, we mainly use a Pohozaev identity of the solution $u_{\lambda}$ from scaling, blow-up analysis and
a contradiction argument.

Let $u_\lambda^{(1)}$ and $u_\lambda^{(2)}$ be two solutions to problem \eqref{quan}.
Let us assume that they concentrate at the same point $0$. We denote

\begin{equation*}	c_\lambda^{(l)}:=u^{(l)}_{\lambda}(0)
=\max_{\overline{B_{\delta}(0)}}u^{(l)}_{\lambda}(x)  ~\mbox{for some small fixed}~\delta>0,
\end{equation*}
and
$$
r^{(l)}_{\lambda}:=2\sqrt{2}\big(\lambda e^{(c_\lambda^{(l)})}\big)^{-1/2}, ~~\mbox{for} ~~l=1,2.
$$
Firstly, we have a basic result on $\frac{r^{(1)}_{\lambda}}{r^{(2)}_{\lambda}}.$
\medskip

\begin{lemma}\label{theta12}
	It holds
	\begin{equation*}\label{llsb}
		\frac{r^{(1)}_{\lambda}}{r^{(2)}_{\lambda}}=1+o_{\lambda}\Big(\frac{1}{g_{c_{\lambda}}}\Big),
	\end{equation*}
where $o_{\lambda}(1)$ denotes a quality which goes to zero as $\lambda\rightarrow0$ and $g_{c_{\lambda}}$
is defined in Lemma \ref{lem-10-7-2}.
\end{lemma}

\begin{proof}
Firstly, we have
\begin{equation}\label{5.1-1}
\frac{r^{(1)}_{\lambda}}{r^{(2)}_{\lambda}}= e^{-\frac{1}{2}\big(c^{(1)}_\lambda-c^{(2)}_\lambda\big)}.
\end{equation}
From \eqref{lambda_gamma}, we find
	\begin{equation}\label{5.1-3}
		\begin{split}
			c^{(1)}_\lambda-c^{(2)}_\lambda=o_{\lambda}\Big(\frac{1}{g_{c_{\lambda}}}\Big).
		\end{split}
	\end{equation}
	Hence we can deduce from \eqref{5.1-1} and \eqref{5.1-3} that
	\begin{equation*}
		\frac{r^{(1)}_{\lambda}}{r^{(2)}_{\lambda}}=1+o_{\lambda}\Big(\frac{1}{g_{c_{\lambda}}}\Big).
	\end{equation*}
\end{proof}

In the following, we will consider the same quadratic form already introduced in \eqref{P}
\begin{equation}\label{p1uv}
	\begin{split}
		P^{(1)}(u,v):=&- 2\delta r^{(1)}_{\lambda}\int_{\partial B_{\delta r^{(1)}_{\lambda}}} \big\langle \nabla u ,\nu\big\rangle \big\langle \nabla v,\nu\big\rangle  d\sigma + \delta r^{(1)}_{\lambda}  \int_{\partial B_{\delta r^{(1)}_{\lambda}}} \big\langle \nabla u , \nabla v \big\rangle  d\sigma.
	\end{split}
\end{equation}

Noting that if $u$ and $v$ are harmonic in $ B_{\delta r^{(1)}_{\lambda}}\backslash \{0\}$, then by Lemma \ref{indep_d}, we know that $P^{(1)}(u,v)$
is independent of $\delta r^{(1)}_{\lambda} \in (0,d].$

Now if $u_{\lambda}^{(1)}\not \equiv u_{\lambda}^{(2)}$ in $B_{\delta r^{(1)}_{\lambda}}$, we set
\begin{equation}\label{eta-def}
	\xi_{\lambda}:=\frac{u_{\lambda}^{(1)}-u_{\lambda}^{(2)}}
	{\|u_{\lambda}^{(1)}-u_{\lambda}^{(2)}\|_{L^{\infty}(B_{\delta r^{(1)}_{\lambda}})}},
\end{equation}
then $\xi_{\lambda}$ satisfies $\|\xi_{\lambda}\|_{L^{\infty}(B_{\delta r^{(1)}_{\lambda}})}=1$ and
\begin{equation}\label{eta-equa}
	\begin{split}
		- \Delta \xi_{\lambda}&=-\frac{\Delta u_{\lambda}^{(1)}-\Delta u_{\lambda}^{(2)}}
		{\|u_{\lambda}^{(1)}-u_{\lambda}^{(2)}\|_{L^{\infty}(B_{\delta r^{(1)}_{\lambda}})}}\\
		&=\frac{u_{\lambda}^{(1)} -  u_{\lambda}^{(2)}}
		{\|u_{\lambda}^{(1)}-u_{\lambda}^{(2)}\|_{L^{\infty}(B_{\delta r^{(1)}_{\lambda}})}}\frac{1}{(1-|x|^{2})^{2}}
		+
		\frac{\lambda \big(e^{(u_{\lambda}^{(1)})} -  e^{(u_{\lambda}^{(2)})}\big)}
		{\|u_{\lambda}^{(1)}-u_{\lambda}^{(2)}\|_{L^{\infty}(B_{\delta r^{(1)}_{\lambda}})}}\\
		&
		=:\Big[\frac{1}{(1-|x|^{2})^{2}}
		+
		E_{\lambda}\Big]\xi_{\lambda},
	\end{split}
\end{equation}
where
\begin{equation}\label{Dlambda-def}
	E_{\lambda}(x):= \lambda   \displaystyle\int_{0}^1
	 e^{ F_t(x)}  dt,
\end{equation}
with $F_t(x):=tu_{\lambda}^{(1)}(x)+(1-t)u_{\lambda}^{(2)}(x)$.
\medskip

\begin{lemma}\label{theta-D}
	There holds
	\begin{equation}\label{Dlambda-equa}
		\big(r^{(1)}_{\lambda}\big)^{2}\Big[\frac{
			1}{(1-|r^{(1)}_{\lambda}x|^{2})^{2}} +E_{\lambda}\big(r^{(1)}_{\lambda}x\big)\Big] =
		8e^{\eta_{0}(x)} \Big(1+ o_{\lambda}(\frac{1}{g_{c_{\lambda}}}) \Big),
	\end{equation}
\[\]
	uniformly on compact sets.
	Particularly, we have

	\begin{equation}\label{Dlambda-lim}
		\big(r^{(1)}_{\lambda}\big)^{2}\Big[\frac{
			1}{(1-|r^{(1)}_{\lambda}x|^{2})^{2}} +E_{\lambda}\big(r^{(1)}_{\lambda}x\big)\Big] \rightarrow
		8e^{\eta_{0}(x)} ~  ~\text{in}~C_{loc}\big(\mathbb{R}^2\big).
	\end{equation}
\end{lemma}
\begin{proof}
	Firstly,  by Lemma \ref{lem-10-7-3} noting that $r^{(1)}_{\lambda}=o\big((c^{(1)}_{\lambda})^{-\frac{1}{2}}\big)=o_{\lambda}(1),$ we have
	\begin{equation}\label{5.5-add1}
		\begin{split}
			\frac{1}{(1-|r^{(1)}_{\lambda}x|^{2})^{2}}=
			1+2|r^{(1)}_{\lambda}x|^{2}+O(|r^{(1)}_{\lambda}x|^{4})
			=1+o_{\lambda}(1),
		\end{split}
	\end{equation}
	which implies that
	\begin{equation}\label{7-11-add1}
		\begin{split}			\frac{\big(r^{(1)}_{\lambda}\big)^{2}}{(1-|r_{\lambda}x|^{2})^{2}}			
=\big(r^{(1)}_{\lambda}\big)^{2}(1+o_{\lambda}(1))=o_{\lambda}\big(\frac{1}{g_{c_{\lambda}}}\big).
		\end{split}
	\end{equation}
Also, we have
\begin{equation}\label{5.5-1}
\begin{split}
\big(r^{(1)}_{\lambda}\big)^{2}
E_{\lambda}\big(r^{(1)}_{\lambda}x\big)
=
\lambda\big(r_\lambda^{(1)}\big)^2
 \displaystyle\int_{0}^1
e^{F_t(r^{(1)}_{\lambda}x)}dt
\end{split}
\end{equation}
and
\begin{equation}\label{5.5-2}
\begin{split}
F_t\big(r^{(1)}_{\lambda}x\big )
&= tu_\lambda^{(1)}
\big(r^{(1)}_{\lambda}x\big )+(1-t) u_\lambda^{(2)}\big(r^{(1)}_{\lambda}x\big )
\\
&= t\big(c^{(1)}_\lambda+\eta^{(1)}_\lambda(x)\big)
			+\big(1-t\big) \bigg(c^{(2)}_\lambda+\eta^{(2)}_\lambda\Big(\frac{ r^{(1)}_\lambda }
				{r^{(2)}_\lambda}x\Big)
\\
&=c^{(1)}_\lambda+\eta^{(1)}_\lambda(x)+\big(1-t\big) f_\lambda(x),
\end{split}
\end{equation}
where
$$
f_\lambda(x):
=c^{(2)}_\lambda-c^{(1)}_\lambda+\eta^{(2)}_\lambda\Big(\frac{ r^{(1)}_\lambda }
{r^{(2)}_\lambda}x\Big)
-\eta^{(1)}_\lambda\big(x\big).
$$
	Moreover, by  \eqref{lambda_gamma} and Lemma \ref{theta12}, direct computation gives

\begin{equation}\label{6-8-1}
\begin{split}
f_\lambda(x)
&= o_{\lambda}\Big(\frac{1}{g_{c_{\lambda}}}\Big)+ \eta_\lambda^{(2)}(x)+o_{\lambda}\Big(\frac{1}{g_{c_{\lambda}}}\Big) -\eta^{(1)}_\lambda\big(x\big)
 = o_{\lambda}\Big(\frac{1}{g_{c_{\lambda}}}\Big)
 +O\left(\big|\eta^{(2)}_\lambda(x)-\eta^{(1)}_\lambda(x)\big|\right).
\end{split}
\end{equation}
It follows form Lemma \ref{lem-10-7-2} that

\begin{equation}\label{5.5-3}
\begin{split}
f_\lambda(x) &= o_{\lambda}\Big(\frac{1}{g_{c_{\lambda}}}\Big)
+O\big(\big|\eta^{(2)}_\lambda(x)-\eta^{(1)}_\lambda(x)\big|\big)\\
&=o_{\lambda}\Big(\frac{1}{g_{c_{\lambda}}}\Big)+O\Big(\big|\frac{w^{(2)}_\lambda(x)-w^{(1)}_\lambda(x)}{g_{c_{\lambda}}}\big|\Big) = o_{\lambda}\Big(\frac{1}{g_{c_{\lambda}}}\Big).
\end{split}
\end{equation}
\[\]
Hence we can deduce from \eqref{lambda_gamma}, \eqref{5.5-2} and \eqref{5.5-3} that

\begin{equation*}
\begin{split}
e^{ F_t (r^{(1)}_{\lambda}x)}
= e^{c_\lambda+\eta^{(1)}_\lambda(x)
+(1-t) f_\lambda(x)}= e^{c_\lambda+\eta^{(1)}_\lambda(x)} \Big(1+o_{\lambda}\Big(\frac{1}{g_{c_{\lambda}}}\Big)\Big),
\end{split}
\end{equation*}
\[\]
which implies
\begin{equation}\label{5.5-4}
\begin{split}
 \big(r_\lambda^{(1)}\big)^2 \int^1_0
 e^{ F_t(r^{(1)}_{\lambda}x) } dt
=& \big(r_\lambda^{(1)}\big)^2 \int^1_0 e^{c^{(1)}_\lambda+\eta^{(1)}_\lambda(x)} \Big(1+o_{\lambda}\Big(\frac{1}{g_{c_{\lambda}}}\Big)\Big) dt
 \\
=& \frac{8}{\lambda } e^{\eta^{(1)}_\lambda(x)}\Big(1+o_{\lambda}\Big(\frac{1}{g_{c_{\lambda}}}\Big)\Big).
\end{split}
\end{equation}
\[\]
Hence, by \eqref{5.5-3} we get

\begin{align}\label{5.5-6}
\begin{split}
\big(r^{(1)}_{\lambda}\big)^{2}
E_{\lambda}\big(r^{(1)}_{\lambda}x\big)
=8 e^{\eta^{(1)}_\lambda(x)}
\Big(1+o_{\lambda}\Big(\frac{1}{g_{c_{\lambda}}}\Big)\Big).
\end{split}
\end{align}
\[\]
Then, from \eqref{7-11-add1},  \eqref{5.5-6} and Lemma \ref{lem-10-7-1}, we derive \eqref{Dlambda-equa}. Also \eqref{Dlambda-lim} can be obtained from the above computations and Lemma \ref{lem-10-7-1}.
\end{proof}

Now applying the blow up analysis, we establish an estimate on $\xi_\lambda$.
\begin{proposition}\label{prop_tildeeta}
	Let $\widetilde{\xi}_{\lambda}(x):=\xi_{\lambda}\big(r^{(1)}_{\lambda}x\big)$, where $\xi_{\lambda}$ is defined in \eqref{eta-def}. Then by taking
	a subsequence if necessary, we have
	\begin{equation}\label{tildeeta-lim}
		\widetilde{\xi}_{\lambda}(x) \to \widetilde{\alpha}_{0} \frac{1-|x|^2}{1+|x|^2}
		~~~\mbox{in}~~~C^1_{loc}\big(\mathbb{R}^2\big),~\mbox{as}~\lambda \to 0,
	\end{equation}
	where $\widetilde{\alpha}_{0}$ is some constant.
\end{proposition}

\begin{proof}
	Since $|\widetilde{\xi}_{\lambda}|\leq 1$,
	by Lemma \ref{theta-D} and the standard elliptic regularity theory, we find that
	\[
	\widetilde{\xi}_{\lambda}(x) \in C^{1,\gamma}\big(B_R\big)~\mbox{and}~
	\|\widetilde{\xi}_{\lambda}\|_{C^{1,\gamma}(B_R)} \leq C
	\]
\[\]
	for any fixed large $R$ and some $\gamma \in (0,1)$, where $C$ is independent of $\lambda$. Then there exists a subsequence (still denoted by $\widetilde{\xi}_{\lambda}$) such that

	$$\widetilde{\xi}_{\lambda}(x)\to \xi_0(x)~~\mbox{in}~C^1\big(B_R\big).$$
\[\]
	By \eqref{eta-equa}, it is easy to check that $\widetilde{\xi}_{\lambda}$ satisfies
	\begin{equation*}
		-\Delta \widetilde{\xi}_{\lambda} =\big(r^{(1)}_{\lambda}\big)^{2}
		\Big[\frac{1 }{(1-|r^{(1)}_{\lambda}x|^{2})^{2}}
		+
		E_{\lambda}\big(r^{(1)}_{\lambda}x\big) \Big]\widetilde{\xi}_{\lambda}    \,\,\,\,\,~\mbox{in}~B_{\delta}.
	\end{equation*}
\[\]
	Then by Lemma \ref{theta-D},
	letting $\lambda \to 0$, we find that $\xi_0$ satisfies
	\begin{equation*}
		-\Delta \xi_0 = 8e^{\eta_{0}} \xi_0    \,\,\,\, \mbox{in} \,\,\,\,   \mathbb{R}^2.
	\end{equation*}
\[\]
	Since $u_{\lambda}^{(l)}(x)(l=1,2)$ is radial, we know that $\xi_{\lambda}(x)$
	is radial which implies that $\widetilde{\xi}_{\lambda} $ is also radial.
	Hence, by Lemma \ref{lem3.1}, we have
	\[
	\xi_0(x)= \widetilde{\alpha}_{0} \frac{1-|x|^2}{1+|x|^2},
	\]
\[\]
	where $\widetilde{\alpha}_{0}$ is some constant.
	
\end{proof}

\begin{proposition}
	For any small fixed constant $\delta>0$, we have the following local Pohozaev identity
\begin{equation}\label{p1_ueta}
\begin{split}
P^{(1)}\big(u_\lambda^{(1)}+u_\lambda^{(2)},\xi_{\lambda}\big)
&=
2\delta r^{(1)}_{\lambda}\int_{\partial  B_{\delta r^{(1)}_{\lambda}}}
\frac{\widetilde{F}_{\lambda}\xi_{\lambda}}{(1-|x|^{2})^{2}} d\sigma
+4\int_{ B_{\delta r^{(1)}_{\lambda}}} \frac{\widetilde{F}_{\lambda}\xi_{\lambda}(1+|x|^{2})}{(1-|x|^{2})^{3}} d\sigma\\
&\quad+2\delta r^{(1)}_{\lambda}\int_{\partial  B_{\delta r^{(1)}_{\lambda}}}
E_{\lambda}\xi_{\lambda} d\sigma
-4\int_{ B_{\delta r^{(1)}_{\lambda}}}
E_{\lambda}\xi_{\lambda} dx,
\end{split}
\end{equation}
where $P^{(1)}$ is the quadratic form in \eqref{p1uv},
	$\nu=\big(\nu_{1},\nu_2\big)$ is the unit  outward normal of $\partial  B_{\delta r^{(1)}_{\lambda}}(0)$ and
\begin{equation}\label{def_tildeE}
\widetilde{F}_{\lambda}(x):= \displaystyle\int_{0}^1
\big(tu_{\lambda}^{(1)}(x)+(1-t)u_{\lambda}^{(2)}(x)\big) dt
\end{equation}
and $E_{\lambda}$ is the same as \eqref{Dlambda-def}.
\end{proposition}

\begin{proof}
By \eqref{puu}, we have
\begin{equation*}
\begin{split}
& P^{(1)}\big(u_\lambda^{(1)}+u_\lambda^{(2)},\xi_{\lambda}\big)
\\
=&~\frac{1}{\|u_\lambda^{(1)}-u_\lambda^{(2)}\|_{L^{\infty}(B_{\delta r_{\lambda}}(0))}} \left(P^{(1)}\big(u_\lambda^{(1)},u_\lambda^{(1)}\big)- P^{(1)}\big(u_\lambda^{(2)},u_\lambda^{(2)}\big)\right)
\\
=&~\frac{1}{\|u_\lambda^{(1)}-u_\lambda^{(2)}\|_{L^{\infty}(B_{\delta r_{\lambda}})}}
\Bigg[ \delta r^{(1)}_{\lambda} \int_{\partial B_{\delta r^{(1)}_{\lambda}}}
\frac{(u^{(1)}_\lambda)^2 - (u^{(2)}_\lambda)^2 }{(1-|x|^{2})^{2}}  \,d\sigma
\\
&+2 \int_{B_{\delta r^{(1)}_{\lambda}}}\big((u^{(1)}_\lambda)^2 - (u^{(2)}_\lambda)^2\big)
\frac{ 1+|x|^{2}}{(1-|x|^{2})^{3}}\,dx
\\
&+
2\lambda\delta r^{(1)}_{\lambda}
\int_{\partial B_{\delta r^{(1)}_{\lambda}}}
\big(e^{u^{(1)}_\lambda}-e^{u^{(2)}_\lambda}\big)\,d\sigma
-4\lambda \int_{B_{\delta r^{(1)}_{\lambda}}}  \big( e^{u^{(1)}_\lambda}-e^{u^{(2)}_\lambda} \big) dx\Bigg]
\\
=&~ 2\delta r^{(1)}_{\lambda} \int_{\partial B_{\delta r^{(1)}_{\lambda}}} \frac{\widetilde{F}_{\lambda}\xi_{\lambda} }{(1-|x|^{2})^{2}}\,d\sigma
+4\int_{B_{\delta r^{(1)}_{\lambda}}} \frac{\widetilde{F}_{\lambda}\xi_{\lambda}(1+|x|^{2})}{(1-|x|^{2})^{3}}  \,dx
\\&+2\delta r^{(1)}_{\lambda} \int_{\partial B_{\delta r^{(1)}_{\lambda}}} E_{\lambda}\xi_{\lambda}\, d\sigma
-4\int_{B_{\delta r^{(1)}_{\lambda}}} E_{\lambda}\xi_{\lambda}\,dx.
\end{split}
\end{equation*}
	
\end{proof}

\begin{proposition}\label{prop-A}
	Let $\widetilde{\alpha}_{0}$ be the constant in \eqref{tildeeta-lim}. Then
	\begin{equation*}
		\widetilde{\alpha}_{0}=0.
	\end{equation*}
\end{proposition}

\begin{proof}
First, by Lemma \ref{lem-10-7-2} we have

\begin{equation}\label{5.9-add4}
\begin{split}
&LHS \  \text{of} ~\eqref{p1_ueta}
\\
&=-2\delta r_{\lambda}^{(1)}\int_{\partial B_{\delta r^{(1)}_\lambda }}
\langle \nabla (u_{\lambda}^{(1)}+u_{\lambda}^{(2)}), \nu\rangle \langle \nabla \xi_{\lambda},\nu\rangle d\sigma
+\delta r_{\lambda}^{(1)}\int_{\partial B_{\delta r^{(1)}_\lambda }}\langle \nabla (u_{\lambda}^{(1)}+u_{\lambda}^{(2)}), \nabla \xi_{\lambda}\rangle  d\sigma
\\
&=-2\delta \int_{\partial B_{\delta }}
\langle \nabla [(u_{\lambda}^{(1)}+u_{\lambda}^{(2)})(r_{\lambda}^{(1)}x)], \nu\rangle \langle \nabla \widetilde{\xi}_{\lambda},\nu\rangle d\sigma
+\int_{\partial B_{\delta  }}\langle \nabla [(u_{\lambda}^{(1)}+u_{\lambda}^{(2)})(r_{\lambda}^{(1)}x)], \nabla \widetilde{\xi}_{\lambda}\rangle  d\sigma
\\
&=-\delta \int_{\partial B_{\delta  }}\langle \nabla [(u_{\lambda}^{(1)}+u_{\lambda}^{(2)})(r_{\lambda}^{(1)}x)], \nabla \widetilde{\xi}_{\lambda}\rangle  d\sigma\\
&=-\delta \int_{\partial B_{\delta  }}
\big(2\nabla \eta_{0}+\nabla (\frac{w_{\lambda}^{(1)}}{g_{c^{(1)}_{\lambda}}}+\frac{w_{\lambda}^{(2)}}{g_{c^{(2)}_{\lambda}}}\big)
\nabla \Big(\tilde{\alpha}_{0}\frac{1-|x|^{2}}{1+|x|^{2}}+o_{\lambda}(1)\Big)   d\sigma\\
&=-64\pi\frac{\delta^{4}}{(1+\delta^{2})^{3}}\widetilde{\alpha}_{0}
-\frac{\widetilde{A}_{1}}{g_{c^{(1)}_{\lambda}}}\widetilde{\alpha}_{0}+o_{\lambda}\Big(\frac{1}{g_{c^{(1)}_{\lambda}}}\Big),
\end{split}
\end{equation}
where
$$
\widetilde{A}_{1}=2\delta\int_{\partial B_{\delta}}\nabla w_{0}\cdot \nabla \frac{1-|x|^{2}}{1+|x|^{2}} d\sigma.
$$
Similar to the proof of \eqref{5.5-4}, we have
\begin{equation}\label{tilde_D-equa}
\begin{split}
\big(r^{(1)}_\lambda\big)^2 E_\lambda\big(r^{(1)}_\lambda x \big)
&= \lambda \big(r^{(1)}_\lambda\big)^2 \int_{0}^{1}
e^{F_t(r^{(1)}_{\lambda}x)} dt
\\
&=\lambda \frac{8}{\lambda}e^{\eta^{1}_{\lambda}(x)}\Big(1+o_{\lambda}\Big(\frac{1}{g_{c^{(1)}_{\lambda}}}\Big)\Big)\\
&= 8e^{\eta_{0}(x)}\Big( 1+\frac{w_{0}}{g_{c^{(1)}_{\lambda}}}+o_{\lambda}\Big(\frac{1}{g_{c^{(1)}_{\lambda}}}\Big) \Big) ~ ~~~~~~~ \mbox{uniformly~on~compact~sets}
\end{split}
\end{equation}
and
\begin{equation}\label{10-10-1}
\begin{split}
\widetilde{F}_{\lambda}(r_{\lambda}^{(1)}x)
&=\int_{0}^{1}[tu_{\lambda}^{(1)}(r_{\lambda}^{(1)}x)+(1-t)u_{\lambda}^{(2)}(r_{\lambda}^{(1)}x)]dt\\
&=c_{\lambda}^{(1)}\int_{0}^{1}\Big(1+\frac{\eta_{\lambda}^{(1)}(x)}{c_{\lambda}^{(1)}}+o\big(\frac{1}{c_{\lambda}^{(1)}g_{c^{(1)}_{\lambda}}}\big)\Big)dt\\
&=c_{\lambda}^{(1)}\Big(1+\frac{\eta_{\lambda}^{(1)}(x)}{c_{\lambda}^{(1)}}
+o_{\lambda}\Big(\frac{1}{g_{c^{(1)}_{\lambda}}}\Big)\Big).
\end{split}
\end{equation}
Applying \eqref{10-10-1}, \eqref{5.5-add1} and Proposition \ref{prop_tildeeta}, we have
\begin{equation}\label{7-11-6.6-1}
\begin{split}
	&2\delta r_{\lambda}^{(1)}
\int_{\partial B_{\delta r^{(1)}_\lambda }}
\frac{\widetilde{F}_{\lambda}(x)\xi_{\lambda}(x)}{(1-|x|^{2})^{2}} d\sigma\\
&=2\delta (r_{\lambda}^{(1)})^{2}\int_{\partial B_{\delta  }}
\frac{\widetilde{F}_{\lambda}(r_{\lambda}^{(1)}x)
\widetilde{\xi}_{\lambda}(x)}{(1-|r_{\lambda}^{(1)}x|^{2})^{2}} d\sigma
\\
&=2\delta (r_{\lambda}^{(1)})^{2}c^{(1)}_{\lambda}
\int_{\partial B_{\delta }}
\Big(1+\frac{\eta_{\lambda}^{(1)}}{c^{(1)}_{\lambda}}
			+o_{\lambda}\Big(\frac{1}{c^{(1)}_{\lambda}g_{c^{(1)}_{\lambda}}}\Big)\Big)
			\Big(\widetilde{\alpha}_{0} \frac{1-|x|^2}{1+|x|^2}+o_{\lambda}(1)\Big) d\sigma
\\
&=\frac{1}{g_{c^{(1)}_{\lambda}}}\frac{4\pi \delta^{2}(1-\delta^{2})}{1+\delta^{2}}\widetilde{\alpha}_{0}+o_{\lambda}\Big(\frac{1}{g_{c^{(1)}_{\lambda}}}\Big)
=:\frac{\widetilde{A}_{2}}{g_{c^{(1)}_{\lambda}}}\widetilde{\alpha}_{0}+o_{\lambda}\Big(\frac{1}{g_{c^{(1)}_{\lambda}}}\Big)
\end{split}
\end{equation}
and
\begin{equation}\label{7-11-6.6-2}
\begin{split}
&4\int_{ B_{\delta r^{(1)}_\lambda }}
\frac{\widetilde{F}_{\lambda}(x)\xi_{\lambda}(x)(1+|x|^{2})}
{(1-|x|^{2})^{3}} dx
\\&=4 (r_{\lambda}^{(1)})^{2}c^{(1)}_{\lambda}\int_{ B_{\delta  }}\Big(1+\frac{\eta_{\lambda}^{(1)}}{c^{(1)}_{\lambda}}
+o_{\lambda}\Big(\frac{1}{c^{(1)}_{\lambda}g_{c^{(1)}_{\lambda}}}\Big)\Big)
\frac{1+|r_{\lambda}x|^{2}}{(1-|r_{\lambda}x|^{2})^{3}}\Big( \widetilde{\alpha}_{0} \frac{1-|x|^2}{1+|x|^2}+
o_{\lambda}(1)\Big)  dx
\\&=\frac{8\pi }{g_{c^{(1)}_{\lambda}}}\Big[\ln(1+\delta^{2})-\frac{\delta}{2}\Big]\widetilde{\alpha}_{0}+o_{\lambda}\Big(\frac{1}{g_{c^{(1)}_{\lambda}}}\Big)
:=\frac{\widetilde{A}_{3}}{g_{c^{(1)}_{\lambda}}}\widetilde{\alpha}_{0}+o_{\lambda}\Big(\frac{1}{g_{c^{(1)}_{\lambda}}}\Big).
\end{split}
\end{equation}
	
Also, applying \eqref{tilde_D-equa} and Proposition \ref{prop_tildeeta}, we have
\begin{equation}\label{5.9-add6}
\begin{split}
&2\delta r_{\lambda}^{(1)}\int_{\partial B_{\delta r^{(1)}_\lambda }}
E_{\lambda}(x)\xi_{\lambda}(x) d\sigma
\\
&=
2\delta (r_{\lambda}^{(1)})^{2}\int_{\partial B_{\delta  }}
E_{\lambda}(r_{\lambda}^{(1)}x)\widetilde{\xi}_{\lambda}(x)
d\sigma
\\&
=16\delta \int_{\partial B_{\delta }}
e^{\eta_{0}(x)}\Big(1+\frac{w_{0}}{g_{c^{(1)}_{\lambda}}}+o_{\lambda}\big(\frac{1}{c^{(1)}_{\lambda}g_{c^{(1)}_{\lambda}}}\big)\Big)
\Big(\widetilde{\alpha}_{0} \frac{1-|x|^2}{1+|x|^2}+o_{\lambda}(1)\Big) d\sigma
\\
&=\frac{32\pi \delta^{2}(1-\delta^{2}) }{(1+\delta^{2})^{3}}\widetilde{\alpha}_{0}
+\frac{\widetilde{\alpha}_{0}}{g_{c^{(1)}_{\lambda}}}\widetilde{A}_{4}+o_{\lambda}\Big(\frac{1}{g_{c^{(1)}_{\lambda}}}\Big),
\,\,\,\text{where}\,\,\widetilde{A}_{4}=16\delta\int_{\partial B_{\delta}}w_{0}\frac{1-|x|^{2}}{(1+|x|^{2})^{3}}d\sigma.
\end{split}
\end{equation}

Similar to \eqref{5.9-add6}, we also have
\begin{equation}\label{5.9-add7}
\begin{split}
&-4\int_{ B_{\delta r^{(1)}_\lambda }}
E_{\lambda}(x)\xi_{\lambda}(x) dx\\
&=-4 (r_{\lambda}^{(1)})^{2}\int_{ B_{\delta  }}
E_{\lambda}(r_{\lambda}^{(1)}x)\widetilde{\xi}_{\lambda}(x)
dx
\\
&=-32\int_{ B_{\delta }}
e^{\eta_{0}(x)}\Big(1+\frac{w_{0}}{g_{c^{(1)}_{\lambda}}}+o_{\lambda}\Big(\frac{1}{g_{c^{(1)}_{\lambda}}}\Big)\Big)\Big( \widetilde{\alpha}_{0} \frac{1-|x|^2}{1+|x|^2}+
			o_{\lambda}(1)\Big)dx
\\&=-\frac{32\pi \delta^{2}}{(1+\delta^{2})^{2}}
\widetilde{\alpha}_{0}+\frac{\widetilde{\alpha}_{0}}{g_{c^{(1)}_{\lambda}}}\widetilde{A}_{5}+o_{\lambda}\Big(\frac{1}{g_{c^{(1)}_{\lambda}}}\Big),\,\,\,\,\text{where}\,\,\,
\widetilde{A}_{5}=-32\int_{ B_{\delta}}w_{0}\frac{1-|x|^{2}}{(1+|x|^{2})^{3}}d\sigma.
\end{split}
\end{equation}

From \eqref{7-11-6.6-1} to \eqref{5.9-add7}, we have
\begin{equation}\label{5.9-add5}
\begin{split}
RHS~ \text{of} ~\eqref{p1_ueta}
&=-64\pi \frac{\delta^{4}}{(1+\delta^{2})^{3}}\widetilde{\alpha}_{0}
+\big(\widetilde{A}_{2}+\widetilde{A}_{3}+\widetilde{A}_{4}+\widetilde{A}_{5}\big)\frac{\widetilde{\alpha}_{0}}{g_{c^{(1)}_{\lambda}}}		+o_{\lambda}\Big(\frac{1}{g_{c^{(1)}_{\lambda}}}\Big).
\end{split}
\end{equation}
Noting that $\widetilde{A}_{1}+\widetilde{A}_{2}+\widetilde{A}_{3}+\widetilde{A}_{4}+\widetilde{A}_{5}\neq 0,$
\eqref{5.9-add4} and \eqref{5.9-add5} yield that
$\widetilde{\alpha}_{0}=0.$
\end{proof}

Now we are in a position to prove Proposition \ref{thuniq}.
\begin{proof}[\textbf{Proof of Proposition \ref{thuniq}}]
	Suppose $u^{(1)}_\lambda\not\equiv u^{(2)}_\lambda$ in $B_{\delta r^{(1)}_{\lambda}}$, and let $\xi_{\lambda}:=\frac{u^{(1)}_\lambda- u^{(2)}_\lambda}{
		\|u^{(1)}_\lambda-u^{(2)}_\lambda\|_{L^\infty{(B_{\delta r^{(1)}_{\lambda}})}}}$. We have
	\begin{equation}\label{5.9-add5-1}
		\|\xi_{\lambda}\|_{L^\infty{(B_{\delta r^{(1)}_{\lambda}})}}=1.
	\end{equation}
	Taking $\widetilde{\xi}_{\lambda}(x):=
	\xi_{\lambda}\big(r^{(1)}_{\lambda}x\big)$, by Propositions \ref{prop_tildeeta} and \ref{prop-A}, we have
	\begin{equation}\label{loc_o}
		\|\widetilde{\xi}_{\lambda}\|_{L^{\infty}(B_R)}=o_{\lambda}\big(1\big)~\mbox{for~any}~R>0,
	\end{equation}
	which implies that
	\begin{equation}\label{5.9-add5-2}
		\|\xi_{\lambda}\|_{L^\infty{(B_{\delta r^{(1)}_{\lambda}})}}=o_{\lambda}\big(1\big).
	\end{equation}
	We get a contradiction from \eqref{5.9-add5-1} and \eqref{5.9-add5-2}. Therefore, we infer that $u^{(1)}_\lambda\equiv u^{(2)}_\lambda$ in $B_{\delta r^{(1)}_{\lambda}}.$
\end{proof}

Finally we prove Theorem \ref{unique}.
\begin{proof}[\textbf{Proof of Theorem \ref{unique}}]
	By Proposition \ref{thuniq}, we obtain that $u^{(1)}_\lambda(0)= u^{(2)}_\lambda(0).$ Then applying Cauchy-initial uniqueness for ODE, we prove that $u^{(1)}_\lambda(x)\equiv u^{(2)}_\lambda(x)$ in $B_{1}.$
\end{proof}
\appendix

\section{Some known results and basic preliminaries}\label{sa}
\vskip 0.2cm
In this section, we will give some known results and some
preliminaries which are used in section 5.
Recall that
\begin{equation*}
-\Delta \eta_{0}(x)=8 e^{\eta_{0}(x)},\,\,\,\,x\in \R^{2},
\end{equation*}
where
\begin{equation}\label{elta-0}
\eta_{0}(x)=-2\ln(1+|x|^{2}).
\end{equation}
\[\]
\begin{lemma}\label{lem3.1}(Lemma 4.3, \cite{EG2004})
Let $\eta_{0}$ be the function defined in \eqref{elta-0} and $v\in C^2(\mathbb{R}^2)\cap L^{\infty}(\mathbb{R}^2)$ be a solution of the following problem
\begin{equation*}
\begin{cases}
-\Delta v=8e^{\eta_{0}}v  ~\mbox{in}~\mathbb{R}^2,\\[1mm]
\displaystyle\int_{\mathbb{R}^2}|\nabla v|^2dx<\infty.
\end{cases}
\end{equation*}
Then it holds
\begin{equation*}
v(x)= \alpha_0\frac{1-|x|^2}{1+|x|^2}+\sum^2_{i=1}{\alpha_i}\frac{x_i}{1+|x|^2}
\end{equation*}
with some $\alpha_0,\alpha_1,\alpha_2\in \mathbb{R}$.
\end{lemma}

\medskip

\begin{lemma}\label{lem-10-7-2}
If $c_{\lambda}\rightarrow\infty$ as $\lambda\rightarrow0,$ then we have $w_{\lambda}(r):=g_{c_{\lambda}}(\eta_{\lambda}(r)-\eta_{0}(r))\rightarrow w_{0}(r)=\frac{1}{ (1+r^2)}\big[\frac{5r^2-r^4}{4}-\frac{1}{2}(1-r^2) \int_1^{1+r^2} \frac{\ln s}{1-s} ds-(1+r^2)\ln (1+r^2)
\big]$
in $C_{\rm{loc}}^2(\mathbb{R}),$ where $g_{c_{\lambda}}=\frac{1}{r^{2}_{\lambda}c_{\lambda}}$ and $w_{0}$
is the unique solution to the ODE
\begin{equation}\label{10-7-5}
-\Delta w_{0}=8e^{\eta_{0}}w_{0}+1,\,\,\,w_{0}(0)=0,\,\,\,w'(0)=0.
\end{equation}
\end{lemma}

\begin{proof}
Notice that $\eta_{\lambda}$ satisfies
\begin{equation}\label{10-7-6}
-\Delta \eta_{\lambda}(x)=r^{2}_{\lambda}
\frac{\eta_{\lambda}(x)+c_{\lambda}}{(1-|r_{\lambda}x|^{2})^{2}}
+8e^{\eta_{\lambda}(x)}.
\end{equation}
Using \eqref{10-7-4} and \eqref{10-7-6}, we compute
\begin{equation}\nonumber
\begin{aligned}
			-\Delta w_{\lambda}
			={}&g_{c_{\lambda}}\frac{r_{\lambda}^2
(\eta_{\lambda}+c_{\lambda})}{(1-r_{\lambda}^2|x|^2)^2}
+8e^{\eta_0}g_{c_{\lambda}}\left[e^{(\eta_{\lambda}-\eta_0)}-1\right].
		\end{aligned}
	\end{equation}
By Lemma \ref{lem-10-7-1} we know for every $R>0$, $\eta_{\lambda}(r)-\eta_0(r)=\frac{w_{\lambda}}{g_{c_{\lambda}}}$ goes to zero as $\lambda\to 0$ uniformly for $r\in[0,R]$. By using a Taylor expansion:
\begin{equation}\nonumber	e^{(\eta_{\lambda}-\eta_0)}=1+\frac{w_{\lambda}}{g_{c_{\lambda}}}
+o_{\lambda}(1)\frac{w_{\lambda}}{g_{c_{\lambda}}},
\end{equation}
with $o_{\lambda}(1)\to 0$ as $\lambda\to 0 $ uniformly for $r\in[0,R]$, from Lemma \ref{lem-10-7-3} we get
	\begin{equation}\label{eq10}
		-\Delta w_{\lambda}=1+8e^{\eta_0}(w_{\lambda}+o_{\lambda}(1)w_{\lambda}),
	\end{equation}
	with $o_{\lambda}(1)\to 0$ as $\lambda\to 0$ uniformly for $r\in[0,R]$. Observing that $w_{\lambda}(0)=w_{\lambda}'(0)=0$, from ODE theory it follows that $w_{\lambda}(r)$ is uniformly bounded in $r\in[0,R]$ and by elliptic estimates we have $w_{\lambda}\to w^\ast$ in $C_{\rm{loc}}^2(\mathbb{R}^2)$, where $w^\ast$ satisfies
$$
-\Delta w^\ast=1+e^{8\eta_{0}}w^\ast.
$$
By direct computations , we can check that $w_{0}(r)=\frac{1}{ (1+r^2)}\big[\frac{5r^2-r^4}{4}-\frac{1}{2}(1-r^2) \int_1^{1+r^2} \frac{\ln s}{1-s} ds-(1+r^2)\ln (1+r^2)
\big]$
is a solution of $-w''_{0}-\frac{1}{r}w'_{0}=8e^{\eta_{0}}w_{0}+1$ satisfying $w_{0}(0)=w_{0}'(0)=0.$
By the Cauchy-initial uniqueness for ODE, we know that $w^\ast=w_{0}(r).$
\end{proof}


Similar to Lemma 4.2 in \cite{LW}, we will establish a priori estimate on $\eta_{\lambda}.$
\begin{lemma}\label{lem2.4}
For any $\epsilon \in (0,2)$, there exist $R_{\epsilon}>1$, $C_{\epsilon}>0$ and $\lambda_{\epsilon} >0$ such that for any $\lambda\in (0,\lambda_{\epsilon})$
\begin{equation}\label{2.4.1}
\eta_{\lambda}(y)\leq \big(4-\epsilon\big)\ln \frac{1}{|y|}+C_{\epsilon} ,
\,\, ~\mbox{for}~ y\in B_{\frac{1}{r_{\lambda}}} \backslash B_{2R_\epsilon}.
\end{equation}
\end{lemma}

\begin{proof}
Given $\epsilon>0$, we can choose $R_{\epsilon}>1$ such that

\begin{equation*}
\int_{B_{R_{\epsilon}}}8 e^{\eta_{0}(z)}dz>(8-\epsilon)\pi.
\end{equation*}
\[\]
Then by Fatou's lemma and $\int_{\R^{2}}8 e^{\eta_{0}(z)}\,dz=8\pi,$ we have

\begin{equation*}
\liminf_{\lambda \to 0}
\int_{B_{R_{\epsilon}}} 8 e^{\eta_{\lambda}(z)}\,dz\geq \int_{B_{R_{\epsilon}}}8 e^{\eta_{0}(z)}\,dz,
\end{equation*}
\[\]
which indicates that there exists $\lambda_\epsilon>0$ sufficiently small such that for any $\lambda\in(0,\lambda_\epsilon)$, it holds
\begin{equation}\label{2.4-1}
\int_{B_{R_{\epsilon}}}8 e^{\eta_{\lambda}(z)}\,dz > (8-\epsilon)\pi.
\end{equation}
\[\]
Also by the Green's representation theorem, we can write

\begin{equation*}
\begin{split}
u_\lambda\big(r_\lambda y\big)= \lambda \int_{B_{1}} G\big(\theta_\lambda y,x\big)
 e^{u_{\lambda}(x)}\,dx=8
 \int_{B_{\frac{1}{r_{\lambda}}}}G\big(r_\lambda y,r_\lambda z\big)
e^{\eta_\lambda(z)}dz.
\end{split}
\end{equation*}
Hence, using \eqref{10-7-1}, we find
\begin{align*}
\eta_{\lambda}(y) 
=& -u_\lambda(0)+8 \int_{B_{\frac{1}{r_{\lambda}}}} G(r_\lambda y,r_\lambda z) e^{\eta_\lambda(x)} \,dz \\
=& -8 \int_{B_{\frac{1}{r_{\lambda}}}} G(0,r_{\lambda}z) e^{\eta_\lambda(z)} \,dz
 +8\int_{B_{\frac{1}{r_{\lambda}}}} G\big(r_\lambda y,r_\lambda z\big) e^{\eta_\lambda(z)} \,dz.
\end{align*}
Now from \eqref{green estimate1} we can compute $\eta_\lambda(y)$ as follows:
\begin{align}\label{2.4-2}
\eta_{\lambda}(y)=&\frac{1}{2\pi} \int_{B_{\frac{1}{r_{\lambda}}}}
\ln \frac{|z|}{|z-y|} 8 e^{\eta_\lambda(z)} \,dz +8\int_{B_{\frac{1}{r_{\lambda}}}}[C(r_\lambda y,r_\lambda z)-C(0,r_\lambda z)]\ e^{\eta_\lambda(z)} \,dz\notag
\\=&
\frac{1}{2\pi}\int_{B_{R_\e}}\ln \frac{|z|}{|z-y|} 8 e^{\eta_\lambda(z)} \,dz \notag\\&
 +\frac{1}{2\pi}\int_{B_{\frac{1}{r_{\lambda}}}\backslash B_{R_{\epsilon}} \bigcap \{|z|\leq 2|z-y|\} }\ln \frac{|z|}{|z-y|} 8 e^{\eta_\lambda(z)} \,dz \,dz
\notag\\&
+\frac{1}{2\pi} \int_{B_{\frac{1}{r_{\lambda}}}\backslash B_{R_{\epsilon}} \bigcap \{|z|\geq 2|z-y|\} } \ln  |z|
8 e^{\eta_\lambda(z)} \,dz
\notag\\&
+\frac{1}{2\pi} \int_{B_{\frac{1}{r_{\lambda}}}\backslash B_{R_{\epsilon}}
\bigcap \{|z|\geq 2|z-y|\} } \ln \frac{1}{|z-y|} 8 e^{\eta_\lambda(z)} \,dz
\notag\\&+8\int_{B_{\frac{1}{r_{\lambda}}}}[C(r_\lambda y,r_\lambda z)-C(0,r_\lambda z)]\ e^{\eta_\lambda(z)} \,dz
 \notag\\[1mm]=:& I_1+I_2+I_3+I_4+I_{5}.
\end{align}
\medskip

Firstly, we estimate $I_1$. If $|z|\leq R_{\epsilon}$, then for any $y\in \Omega_\lambda \backslash B_{2R_{\epsilon}}$, we have
$2|z|\leq |y|$ and

\begin{equation*}
\frac{|z|}{|z-y|} \leq \frac{|z|}{|y|-|z|} \leq \frac{|z|}{|y|-\frac{|y|}{2}}\leq \frac{2R_{\epsilon}}{|y|}.
\end{equation*}
Therefore, it holds

\begin{equation*}
I_1\leq
\frac{1}{2\pi}\ln \frac{2R_{\epsilon}}{|y|} \int_{B_{R_{\epsilon}}}8
e^{\eta_\lambda(z)}dz,
\end{equation*}
which together with \eqref{2.4-1} indicates that

\begin{equation}\label{2.4-4}
I_1\leq \Big(4-\frac{\epsilon}{2}\Big) \ln \frac{1}{|y|} + C_1(\epsilon),
\end{equation}
where $C_1(\epsilon)>0$ is a constant dependent on $\epsilon.$

\vskip 0.2cm

\medskip

Next, we estimate the term $I_2$. If $|z|\leq 2|z-y|$, then $\ln \frac{|z|}{|z-y|}\leq
\ln 2$. Hence, from \eqref{2.4-1} and \eqref{10-7-1}, we calculate that

\begin{equation}\label{2.4-5}
I_2\leq \frac{\ln2}{2\pi} \int_{B_{\frac{1}{r_{\lambda}}}\backslash B_{R_{\epsilon}} }
 8 e^{\eta_\lambda(z)}\,dz \leq C_2(\epsilon) ,
\end{equation}
\[\]
where $C_2(\epsilon)>0$ is a constant dependent on $\epsilon$. Similarly, if $|z|\geq 2|z-y|$, we have  $|z|\leq 2|y|$ and

\begin{equation}\label{2.4-6}
I_3\leq \frac{\ln |2y|}{2\pi} \int_{\Omega_\lambda\backslash B_{R_{\epsilon}}}8 e^{\eta_\lambda(z)}\,dz
 \leq \frac{\epsilon}{2} \ln |y| + C_3(\epsilon),
\end{equation}
\[\]
where $C_3(\epsilon)>0$ is a constant dependent on $\epsilon.$
\medskip

Finally, we know that

\begin{equation*}
\begin{split}
& \Omega_\lambda\backslash B_{R_{\epsilon}} \cap \big\{ |z| \geq 2|z-y| \big\} \\
=&~ \Big( \Omega_\lambda\backslash B_{R_{\epsilon}}\cap \big\{2\leq 2|z-y|\leq |z|\big\} \Big) \bigcup
\Big( \Omega_\lambda\backslash B_{R_{\epsilon}} \cap \big\{ 2|z-y|\leq |z| \leq 2 \big\} \Big)  \\
&~ \bigcup \Big( \Omega_\lambda\backslash B_{R_{\epsilon}} \cap \big\{ 2|z-y| \leq 2 \leq |z|\big\} \Big),
\end{split}
\end{equation*}
\[\]
then by H\"older's inequality, for any fixed $\alpha$ (closed to $1^+$), we compute $I_4$:

\begin{equation}\label{2.4-7}
\begin{split}
I_4 =&~
\frac{1}{2\pi} \int_{\Omega_\lambda\backslash B_{R_{\epsilon}}\bigcap \{2|z-y|\leq 2 \leq |z|\}} \ln \frac{1}{|z-y|}
 8e^{\eta_\lambda(z)}\,dz \\
&~ +\frac{1}{2\pi} \int_{\Omega_\lambda\backslash B_{R_{\epsilon}} \bigcap \{2|z-y|\leq |z| \leq 2\}} \ln \frac{1}{|z-y|}
 8e^{\eta_\lambda(z)} \,dz \\
&~ +\frac{1}{2\pi} \int_{\Omega_\lambda\backslash B_{R_{\epsilon}} \bigcap \{2\leq 2|z-y|\leq |z|\} } \ln \frac{1}{|z-y|}
 8e^{\eta_\lambda(z)} \,dz \\
\leq &~ \frac{1}{2\pi} \int_{\Omega_\lambda\backslash B_{R_{\epsilon}}\bigcap \{|z-y|\leq 1\}} \ln \frac{1}{|z-y|}
e^{8\eta_\lambda(z)} \,dz \\
\leq &~ \frac{1}{2\pi} \bigg(\int_{|z-y|\leq 1} \Big| \ln \frac{1}{|z-y|} \Big|^{\frac{\alpha}{\alpha-1}}  \,dz\bigg)^{\frac{\alpha-1}{\alpha}} \cdot \Bigg(\int_{\Omega_\lambda\backslash B_{R_{\epsilon}}} \bigg|8 e^{\eta_\lambda(z)} \bigg|^{\alpha} \,dz\Bigg)^{\frac{1}{\alpha}} \\ \leq &~ C_4(\epsilon),
\end{split}
\end{equation}
\[\]
here we use \eqref{2.4-1} and the fact that
\[
\ln \frac{1}{|z-y|} \leq0,~\mbox{in}~\big\{z\in \Omega_\lambda\backslash B_{R_{\epsilon}}: 2\leq 2|z-y|\leq |z|\big\}.
\]
\[\]
Finally, we estimate

\begin{equation}\label{10-10-2}
\begin{split}
I_5 &\leq C \int_{B_{\frac{1}{r_{\lambda}}}}e^{\eta_{\lambda}(z)}\,dz\leq C.
\end{split}
\end{equation}
Substituting \eqref{2.4-4}, \eqref{2.4-5}, \eqref{2.4-6}, \eqref{2.4-7}  and \eqref{10-10-2} into \eqref{2.4-2}, we obtain that for any
$ \lambda \in (0,\lambda_\epsilon)$ and $y \in \Omega_\lambda \backslash B_{2R_\epsilon}$, there exists $C_\epsilon>0$ such that
\begin{eqnarray*}
w_\lambda(y) & \leq& \big(4-\frac{\epsilon}{2} \big) \ln \frac{1}{|y|} + \frac{\epsilon}{2} \ln |y| + C_1(\epsilon)+C_2(\epsilon)+C_3(\epsilon)+C_4(\epsilon)+C \\
&\leq& (4-\epsilon) \ln \frac{1}{|y|} + C_\epsilon.
\end{eqnarray*}

\end{proof}

Just by the same arguments as those \cite{DT,mm}, we can prove the following result.
\begin{lemma}\label{add-lem1-12-12}
Let $0\leq s \leq s_{\lambda}\leq r^{-\theta}_{\lambda},$ $0<\theta<1$ being small enough
and $\phi:[s,s_{\lambda}]\rightarrow \R$ be given so that $\phi=o(g^{3}_{c_{\lambda}})$
uniformly on $[s,s_{\lambda}].$ Set
$$
\eta:=\eta_{0}+\frac{w_{0}}{g_{c_{\lambda}}}+\frac{h_{0}}{g^{2}_{c_{\lambda}}}+\frac{\phi}{g^{3}_{c_{\lambda}}}
$$
and
\begin{equation}\label{lambda_gamma-12-12}
\Phi_{\lambda}(r,\phi)=:g^{3}_{c_{\lambda}}\Big[\frac{r^{2}_{\lambda}(\eta_{\lambda}+c_{\lambda})}{(1-|r_{\lambda}x|^{2})^{2}}
+8e^{\eta_{\lambda}}+\Delta \eta_{0}+\frac{\Delta w_{0}}{g_{c_{\lambda}}}+\frac{\Delta h_{0}}{g^{2}_{c_{\lambda}}}\Big].
\end{equation}
Then
$$
\Phi_{\lambda}(r,\phi)=8e^{\eta_{0}}\big(\phi+o(1)\phi+O(g^{-1}_{c_{\lambda}}\xi)\phi+O(\xi^{2})\big),\,\,\,
\text{uniformly~for}\,\,r\in [s,s_{\lambda}],
$$
where $\xi(r)=1+r^{2}$ and
$h_{0}$ satisfies
$
-\Delta h_{0}=8e^{\eta_{0}}(w^{2}_{0}+h_{0}),\,\,\,h_{0}(0)=0,\,\,\,h_{0}'(0)=0.
$
\end{lemma}

Applying Lemma \ref{add-lem1-12-12}, we can obtain some estimates.
\begin{lemma}\label{add-lem1-12-13}
There exist $M>0$ and $T>0$ such that
$$
\eta_{\lambda}=\eta_{0}+\frac{w_{0}}{g_{c_{\lambda}}}+\frac{h_{0}}{g^{2}_{c_{\lambda}}}+\frac{\phi_{\lambda}}{g^{3}_{c_{\lambda}}}
$$
with
\begin{equation}\label{lambda_gamma-12-12-1}
|\phi_{\lambda}(r)|\leq M \xi(r),\,\,\text{for}\,\,r\in [0,r^{-\theta}_{\lambda}],\,\,\,|\phi'_{\lambda}(r)|\leq Mr\,\,\,
\text{for}\,\,\,r\in [T,r^{-\theta}_{\lambda}]
\end{equation}
for $\lambda$ (depending on M and T), where $\xi(r)$ is as in Lemma \ref{add-lem1-12-12}.
\end{lemma}

Now we estimate the maximum value $c_{\lambda}$ of the solution $u_{\lambda}.$

\begin{lemma}\label{add-lem1}
There holds

\begin{equation}\label{lambda_gamma}
c_\lambda= A-2\Big(1+\frac{\widetilde{A}}{g_{c_{\lambda}}}\Big)\ln \lambda +\frac{B}{g_{c_{\lambda}}}+o_{\lambda}\Big(\frac{1}{g_{c_{\lambda}}}\Big),
\end{equation}
where $A$ and $B$ are constants independent of $\lambda$ and  $o_{\lambda}(1)$ denotes a quality which goes to zero as $\lambda\rightarrow 0.$
\end{lemma}

\begin{remark}\label{add-rem-1-22}
We would like to point out that in fact we can obtain a little rough estimate $c_\lambda= A-2\ln \lambda +o_{\lambda}(1).$  Furthermore, we can prove $r_{\lambda}=O\big(e^{-\frac{c_{\lambda}}{4}}\big).$
\end{remark}

\begin{remark}\label{add-rem-1-23}
Also, there holds that $\frac{1}{g_{c_{\lambda}}}=O\big(e^{-(\frac{1}{2}-\epsilon)c_{\lambda}}\big),$ where $\epsilon>0$ is small.
\end{remark}

\begin{proof}
By the Green's representation theorem, we have

\begin{equation}\label{4.11-1}
\begin{split}
c_\lambda &= u_\lambda(0)
=\lambda \int_{B_{1}} G(y,0)  e^{u_{\lambda}(y)} dy.
\end{split}
\end{equation}
By scaling and using the properties of Green's function \eqref{green estimate1}, we know

\begin{equation}\label{4.11-2-1}
\begin{split}
\lambda  \int_{B_{1}}  G(y,0) e^{u_{\lambda}(y)} dy
=&\lambda  \int_{B_{s_{\lambda}r_{\lambda}}}  G(y,0) e^{u_{\lambda}(y)} dy
+\int_{B_{1}\setminus B_{s_{\lambda}r_{\lambda}}}G(y,0) e^{u_{\lambda}(y)} dy\\
=&~\lambda r_\lambda^2 \int_{B_{s_{\lambda}}} G\big(r_{\lambda}z,0\big)
		 e^{u_\lambda(r_{\lambda}z)}  dz+o\Big(\frac{1}{g_{c_{\lambda}}}\Big) \\
		=&~8 \int_{B_{s_{\lambda}}} \Big(-\frac{1}{2\pi}\ln |r_{\lambda} z|+\widetilde{C}(r_\lambda z )\Big)e^{\eta_{\lambda}(z)}dz
+o\Big(\frac{1}{g_{c_{\lambda}}}\Big),
\end{split}
\end{equation}
\[\]
since by Lemma \ref{lem2.4} and Remark \ref{add-rem-1-22} we can estimate
\begin{equation}\label{4.11-2}
\begin{split}
&\int_{B_{1}\setminus B_{s_{\lambda}r_{\lambda}}}G(y,0) e^{u_{\lambda}(y)} dy
=~\lambda r_\lambda^2 \int_{B_{\frac{1}{r_{\lambda}}}\setminus B_{s_{\lambda}}} G\big(r_{\lambda}z,0\big)
		 e^{u_\lambda(r_{\lambda}z)}  dz\\
		=&~8 \int_{B_{\frac{1}{r_{\lambda}}}\setminus B_{s_{\lambda}}}G\big(r_{\lambda}z,0\big) e^{\eta_{\lambda}(z)}dz
=O\Big(\int_{B_{\frac{1}{r_{\lambda}}}\setminus B_{s_{\lambda}}}e^{(4-\epsilon)\ln \frac{1}{|y|}+C_{\epsilon}}\Big)\\
=&~O\big(r^{2-\epsilon}_{\lambda}\big)+O\Big(\frac{1}{g^{1+\epsilon}_{c_{\lambda}}}\Big)
=o\Big(\frac{1}{g_{c_{\lambda}}}\Big).
\end{split}
\end{equation}
\[\]
By the expansion of $\eta_{\lambda}(z)$ in $B_{s_{\lambda}}$,
we have
\begin{equation*}
\begin{split}
\int_{B_{s_{\lambda}}}8e^{\eta_{\lambda}(z)}\,dz
&=\int_{B_{s_{\lambda}}}8e^{\eta_{0}(z)+\frac{w_{0}(z)}{g_{c_{\lambda}}}+\frac{h_{\lambda}(z)}{g^{2}_{c_{\lambda}}}}\,dz\\
&=\int_{B_{s_{\lambda}}}8e^{\eta_{0}(z)}\,dz
+\int_{B_{s_{\lambda}}}8e^{\eta_{0}(z)}\Big(\frac{w_{0}(z)}{g_{c_{\lambda}}}+\frac{h_{\lambda}(z)}{g^{2}_{c_{\lambda}}}
+o\big(\frac{w_{0}(z)}{g_{c_{\lambda}}}+\frac{h_{\lambda}(z)}{g^{2}_{c_{\lambda}}}\big)\Big)\,dz\\
&=\int_{\mathbb{R}^{2}}8e^{\eta_{0}(z)}\,dz+\frac{8}{g_{c_{\lambda}}}\int_{\mathbb{R}^{2}}e^{\eta_{0}(z)}w_{0}(z)\,dz
+o_{\lambda}\Big(\frac{1}{g_{c_{\lambda}}}\Big)\\
&=8\pi+\frac{8\widetilde{A}}{g_{c_{\lambda}}}+o_{\lambda}\Big(\frac{1}{g_{c_{\lambda}}}\Big),\,\,\,\text{where}\,\,\,
\widetilde{A}=\int_{\R^{2}}e^{\eta_{0}(z)}w_{0}(z)\,dz.
\end{split}
\end{equation*}
\[\]

Similarly, one can derive that
$$
\int_{B_{s_{\lambda}}}\widetilde{C}(r_\lambda z)8e^{\eta_{\lambda}(z)}\,dz=8\pi \widetilde{C}(0)
+8\widetilde{C}(0)\frac{\widetilde{A}}{g_{c_{\lambda}}}+
+o_\lambda\Big(\frac{1}{g_{c_{\lambda}}}\Big).
$$
\[\]
Hence
\begin{equation*}
\begin{split}
\ln r_\lambda \int_{B_{s_{\lambda}}}8e^{\eta_{\lambda}(z)}\,dz
&=\ln r_\lambda \Big(8\pi+\frac{8\widetilde{A}}{g_{c_{\lambda}}}+o_{\lambda}\Big(\frac{1}{g_{c_{\lambda}}}\Big)\Big)\\
&=\ln\big(2\sqrt{2}(\lambda e^{c_{\lambda}})^{-\frac{1}{2}}\big)\Big(8\pi+\frac{8\widetilde{A}}{g_{c_{\lambda}}}+o_{\lambda}\Big(\frac{1}{g_{c_{\lambda}}}\Big)\Big)
\\
&=\Big(\ln 2\sqrt{2}-\frac{1}{2}\ln \lambda -\frac{1}{2}c_{\lambda}\Big)\Big(8\pi+\frac{8\widetilde{A}}{g_{c_{\lambda}}}+o_{\lambda}\Big(\frac{1}{g_{c_{\lambda}}}\Big)\Big).
\end{split}
\end{equation*}
\[\]
Similarly, it holds
\begin{equation*}
\begin{split}
\int_{B_{s_{\lambda}}}\ln |z| e^{\eta_{\lambda}(z)}\,dz
&=\int_{B_{s_{\lambda}}}\ln |z| e^{\eta_{0}(z)+\frac{w_{0}(z)}{g_{c_{\lambda}}}+\frac{h_{\lambda}(z)}{g^{2}_{c_{\lambda}}}}\,dz\\
&=\int_{\R^{2}}\ln |z| e^{\eta_{0}(z)}\,dz
+\frac{1}{g_{c_{\lambda}}}\int_{\R^{2}}e^{\eta_{0}(z)}\ln|z|w_{0}(z)dz+o_{\lambda}\Big(\frac{1}{g_{c_{\lambda}}}\Big)\\
&=\frac{\widetilde{B}}{g_{c_{\lambda}}}+o_{\lambda}\Big(\frac{1}{g_{c_{\lambda}}}\Big),
\end{split}
\end{equation*}
\[\]
where $\widetilde{B}=\int_{\R^{2}}e^{\eta_{0}(z)}\ln|z|w_{0}(z)dz$ 
and

	\begin{eqnarray*}
		\int_{\R^{2}}\ln |z| e^{\eta_{0}(z)}\,dz
&=&\int_{\R^{2}}\ln |z| e^{-2\ln(1+|z|^{2})}\,dz\\
&=&\int_{0}^{2\pi}d\theta \int_{0}^{+\infty}\frac{r\ln r}{(1+r^{2})^{2}}\,dr
\\
&=&-2\pi \int_{0}^{+\infty}\frac{\frac{1}{t}\ln \frac{1}{t}}{(1+\frac{1}{t})^{2}}\Big(-\frac{1}{t^{2}}\Big)\,dt\\
&=&-2\pi \int_{0}^{+\infty}\frac{t\ln t}{(1+t^{2})^{2}}\,dt,
\end{eqnarray*}
which implies that
$$
\int_{0}^{+\infty}\frac{t\ln t}{(1+t^{2})^{2}}\,dt=0.
$$	
	
Hence we have
\begin{equation}\label{4.11-4}
\begin{split}
& \lambda  \int_{B_{1}} G(y,0) e^{u_{\lambda}(y)}\, dy \\
=&-\frac{1}{2\pi}\Big[\big(\ln2\sqrt{2}-\frac{1}{2}\ln\lambda-\frac{1}{2}c_{\lambda}\big)
\Big(8\pi+\frac{8\widetilde{A}}{g_{c_{\lambda}}}+o_{\lambda}\Big(\frac{1}{g_{c_{\lambda}}}\Big) \Big)+8\frac{\widetilde{B}}{g_{c_{\lambda}}}
+o_{\lambda}\Big(\frac{1}{g_{c_{\lambda}}}\Big)\Big]
\\
&\quad+8\pi \widetilde{C}(0)
+8\widetilde{C}(0)\frac{\widetilde{A}}{g_{c_{\lambda}}}+
+o_\lambda\Big(\frac{1}{g_{c_{\lambda}}}\Big).
\end{split}
\end{equation}
\[\]
Combining \eqref{4.11-2} and \eqref{4.11-4}, we have
\begin{equation}\label{4.11-6}
c_{\lambda}
=A-2\Big(1+\frac{\widetilde{A}}{\pi g_{c_{\lambda}}}\Big)\ln \lambda+\frac{B}{g_{c_{\lambda}}}+o_\lambda\Big(\frac{1}{g_{c_{\lambda}}}\Big),
\end{equation}
where
$$
A=\ln 64-8\pi \widetilde{C}(0)\,\,\text{and}\,\,\,B=\frac{\widetilde{A}}{\pi}\ln 64+\frac{4}{\pi}\widetilde{B}-8 \widetilde{C}(0)\widetilde{A}.
$$
\end{proof}

Recall that 0 is the local maximum point of the solution $u_\lambda$.
Let us define the following quadratic form
\begin{equation}\label{P}
	\begin{split}
		P(u,v):=&- 2d\int_{\partial B_d}\big\langle \nabla u ,\nu\big\rangle \big\langle \nabla v,\nu\big\rangle  d\sigma + d \int_{\partial B_d} \big\langle \nabla u , \nabla v \big\rangle d\sigma,
	\end{split}
\end{equation}
where $u,v\in C^{2}(\overline{B_{1}})$, $d>0$ is a small constant such that $B_{d}\subset B_{1}$ and $ \nu= ( \nu_1, \nu_2)$ is the unit outward normal of $\partial B_d$.
\vskip 0.2cm

By Lemma 2.4 in \cite{LPP-2022}, we have the following property about the above quadratic form.
\begin{lemma}\label{indep_d}
	If $u$ and $v$ are harmonic in $ B_d\backslash \{0\}$, then $P(u,v)$ is independent of $\theta \in (0,d].$
\end{lemma}

Next, we have the following identity about the quadratic form $P$ on the solution $u_\lambda$.
\begin{lemma}\label{lem2.3}
	Let $u_\lambda\in C^2(B_{1})$ be a solution of problem \eqref{quan} and $\delta>0$ is a fixed small constant such that $B_{\delta r_{\lambda}}\subset B_{1}$. Then
\begin{equation}\label{puu}
\begin{split}
P\big(u_\lambda,u_\lambda\big)&=
\delta r_{\lambda}\int_{\partial B_{\delta r_{\lambda}}}\frac{u^2_\lambda}{(1-|x|^{2})^{2}}   d\sigma
+2\int_{B_{\delta r_{\lambda}}} \frac{ u^2_\lambda(1+|x|^{2})} {(1-|x|^{2})^{3}}  dx
+2\lambda \delta r_{\lambda}\int_{\partial B_{\delta r_{\lambda}}} e^{u_\lambda} d\sigma \\&\quad-4\lambda \int_{B_{\delta r_{\lambda}}} e^{u_\lambda} dx.
\end{split}
\end{equation}
	
\end{lemma}

\begin{proof}
Multiplying $\big\langle x, \nabla u_\lambda \big\rangle$ on both sides of equation \eqref{quan} and integrating on $B_{\delta r_{\lambda}}$, we have
\[
\begin{split}
&\frac{1}{2}\int_{\partial B_{\delta r_{\lambda}}} \big\langle x,\nu\big\rangle |\nabla u_\lambda|^2 \, d\sigma
		-\int_{\partial B_{\delta r_{\lambda}}}\frac{\partial u_\lambda}{\partial\nu} \big\langle x, \nabla u_\lambda \big\rangle  \,d\sigma
 \\
=&\frac{1}{2}\int_{\partial B_{\delta r_{\lambda}}}\frac{u^2_\lambda}{(1-|x|^{2})^{2}} \big\langle x,\nu\big\rangle  \,d\sigma +
		\int_{B_{\delta r_{\lambda}}} \frac{u^2_\lambda(1+|x|^{2})}{(1-|x|^{2})^{3}} dx
\\
&+\lambda\int_{\partial B_{\delta r_{\lambda}}} x\cdot \nu e^{u_\lambda} \,d\sigma
		-2\lambda \int_{ B_{\delta r_{\lambda}}} e^{u_\lambda}\,dx,
\end{split}
\]
	which together with \eqref{P} implies \eqref{puu}.
	
\end{proof}

\textbf{Data availability statement}\\
The authors confirm that there is no data used in our manuscript.

\end{document}